\documentclass[dvips]{imsart}

\usepackage{amsthm,amsmath,natbib}
\usepackage{graphics,ifthen}
\usepackage{latexsym}
\usepackage{mathrsfs}
\usepackage{amssymb}

\startlocaldefs

\def\PK{\mbox{\scriptsize PK}}
\def\B{\beta}

\def\Pup{\tau^{(k)}_{m,r}}
\def\Pdown{\delta^{(k)}_{m,r}}
\numberwithin{equation}{section} \theoremstyle{plain}
\newtheorem{thm}{Theorem}[section]
\newtheorem{defin}{Definition}[section]
\newtheorem{prop}{Proposition}[section]
\newtheorem{lem}{Lemma}[section]
\newtheorem{rem}{Remark}[section]
\newtheorem{cor}{Corollary}[section]

\newtheorem{property}{Property}[section]

\newcommand{\indic}{\mathbb{I}}

\newcommand{\fst}{ F_* }

\newcommand {\epo}{ \varepsilon_0}

\newcommand {\epone}{ \varepsilon_1}
\newcommand {\eptwo}{ \varepsilon_2}

\newcommand {\epi}{ \varepsilon_i}

\newcommand {\tjs}{ T_{j}^*}

\newcommand{\br}{\mbox{$\scriptstyle{\rm br}$}}
\newcommand{\ex}{\mbox{$\scriptstyle{\rm ex}$}}

\newcommand{\BB}{B^{\br}}
\newcommand{\Bex}{B^{\ex}}

\newcommand{\la}{\lambda}

\newcommand{\muu}{\mu}
\newcommand{\ggla}{\la^{-\gamma} \Gamma_{\delta}}
\newcommand{\gamal}{\Gamma_{\delta}}

\newcommand{\fbrj}{F_j^{\br}}

\newcommand{\fbro}{F_1^{\br}}
\newcommand{\fbrt}{F_2^{\br}}

\def\qed{\mbox{\rule{0.5em}{0.5em}}}

\endlocaldefs

\begin{document}

\begin{frontmatter}
\title{Gibbs Partitions (EPPF's) Derived From a Stable Subordinator are Fox $H-$ and Meijer $G-$Transforms\protect}
\runtitle{Stable Poisson Kingman}

\begin{aug}
\author{\fnms{Man-Wai} \snm{Ho}\thanksref{t2}\ead[label=e1]{stahmw@nus.edu.sg}
\ead[label=u1,url]{www.stat.nus.edu.sg/$\sim$stahmw}},
\author{\fnms{Lancelot F.} \snm{James}\thanksref{t3}\ead[label=e2]{lancelot@ust.hk}
\ead[label=u2,url]{ihome.ust.hk/$\sim$lancelot} } and
\author{\fnms{John W.} \snm{Lau}
\ead[label=e3]{john.w.lau@googlemail.com}}
\ead[label=u3,url]{http://www.stats.wits.ac.za/john1.html}

\thankstext{t2}{Supported in part by National
University of Singapore research grant R-155-050-067-101
and R-155-050-067-133.}
\thankstext{t3}{Supported in part by the grant HIA05/06.BM03 of the HKSAR } \runauthor{M.-W. Ho, L.~F. James
and J.~W. Lau}

\affiliation{National University of Singapore, Hong Kong
University of Science and Technology and University of
Witwatersrand}

\address{Department of Statistics and Applied Probability \\
National University of Singapore\\
6 Science Drive 2\\
Singapore 117546\\
Republic of Singapore\\
\printead{e1}\\
\printead{u1} }

\address{The Hong Kong University of Science and Technology\\
Department of Information and Systems Management\\
Clear Water Bay, Kowloon\\
Hong Kong\\
\printead{e2}\\
\printead{u2} }

\address{Department of Statistics and Actuarial Science\\
University of Witwatersrand\\
Johannesburg WITS2050\\
South Africa\\
\printead{e3}\\
\printead{u3} }
\end{aug}

\begin{abstract}
This paper derives explicit results for the infinite Gibbs
partitions generated by the jumps of an $\alpha-$stable
subordinator, derived in Pitman~\cite{Pit02, Pit06}. We first show
that for general $\alpha$ the conditional EPPF can be represented as
ratios of Fox $H$-functions, and in the case of rational $\alpha,$
Meijer $G$-functions. This extends results for the known case of
$\alpha=1/2,$ which can be expressed in terms of Hermite functions,
hence answering an open question. Furthermore the results show that
the resulting unconditional EPPF's can be expressed in terms of $H$-
and $G$-transforms indexed by a function $h.$ Hence when $h$ is
itself a $H$- or $G$-function the EPPF is also an $H$- or
$G$-function. An implication, in the case of rational $\alpha,$ is
that one can compute explicitly thousands of EPPF's derived from
possibly exotic special functions. This would also apply to all
$\alpha$ except that computations for general Fox $H$-functions are
not yet available. However, moving away from special functions, we
demonstrate how results from probability theory may be used to
obtain calculations. We show that a forward recursion can be applied
that only requires calculation of the simplest components.
Additionally we identify general classes of EPPF's where explicit
calculations can be carried out using distribution theory.
Specifically what we call the Lamperti class and Beta-Gamma class.
As a special application, we use the latter class to obtain EPPF's
based on mixing distributions derived from the laws of of ranked
functionals of self similar Markovian excursions. The work serves
importantly a dual purpose. One is to obtain explicit calculations
for large classes of EPPF's and hence of use to the growing number
of applications involving combinatorial stochastic processes. The
other, perhaps surprising, is the introduction of new techniques for
calculating explicitly certain Fox $H$-transforms and related
quantities via probabilistic arguments.
\end{abstract}

\begin{keyword}[class=AMS]
\kwd[Primary ]{62G05} \kwd[; secondary ]{62F15}
\end{keyword}

\begin{keyword}
\kwd{beta gamma algebra, Brownian and Bessel processes, Fox
$H$- and Meijer $G$-functions, Lamperti type laws, stable
Poisson Kingman Gibbs partitions.}
\end{keyword}


\end{frontmatter}

\section{Introduction} Let $S_{\alpha},$ for $0<\alpha<1,$ denote a positive
$\alpha$-stable random variable, whose law is specified by
the Laplace transform,
$$
\mathbb{E}[{\mbox e}^{-\lambda S_{\alpha}}]={\mbox
e}^{-\lambda^{\alpha}}
$$
for $\lambda>0,$ and with density denoted as $f_{\alpha}.$ Then,
following~\cite{PPY92, PY97} it is well-known that
$$
S_{\alpha}=\sum_{i=1}^{\infty}J_{i}
$$
where $J_{1}\ge J_{2}\ge \ldots>0$ are the ranked jump sizes of a
stable subordinator
$$
S_{\alpha}(s)=\sum_{i=1}^{\infty}J_{i}\indic(U_{i}\leq s),\quad
0\leq s\leq 1, {\mbox { with }} S_{\alpha}(1)=S_{\alpha},
$$
where the $(U_{i})$ are independent random times distributed
uniformly on $(0,1)$ and also independent of the $(J_{i}).$
Furthermore, the subordinator is characterized by its L\'evy
density,
$$
\rho_{\alpha}(s)=\frac{\alpha}{\Gamma(1-\alpha)} s^{-\alpha-1}
{\mbox { for }}s>0.
$$
Following  Kingman~\cite{Kingman75} and
Perman-Pitman-Yor~\cite{PY92,PPY92, PY97},
Pitman~\cite{Pit02} discussed the laws of the ranked jumps
normalized by their random total mass $S_{\alpha}$ and
further conditioned on $S_{\alpha}=t$, that is,
$$ \mathcal{L}\left(P_{1},P_{2},
\ldots\mid S_{\alpha}=t \right)
$$
where $(P_{i})=(J_{i}/S_{\alpha}).$ The distribution is referred to
as the~(conditional) \emph{Poisson Kingman distribution with L\'evy
density} $\rho_{\alpha}$ and denoted as \PK$(\rho_{\alpha}|t)$.
Furthermore, one can create an infinite number of laws from this
construction by mixing over $t$ with any distribution $\gamma$ on
$(0,\infty).$ The law of the sequence $(P_{i})$ is then referred to
as the \emph{Poisson Kingman distribution with L\'evy density
$\rho_{\alpha}$ and mixing distribution $\gamma,$} denoted as
$$
\PK(\rho_{\alpha},\gamma)=\int_{0}^{\infty}{\mbox{\scriptsize
PK}}(\rho_{\alpha}|t)\gamma(dt).
$$

In this paper we provide explicit calculations and
interpretations for the \emph{exchangeable partition
probability function} (EPPF) which characterizes the law of
the exchangeable random partitions $\Pi_{\infty}=(\Pi_{n})$
on $\mathbb{N}$ generated by the ranked jumps of an
$\alpha$ stable subordinator, that is, the
\PK$(\rho_{\alpha},\gamma)$ partitions. Equivalently, using
\emph{Kingman's paintbox representation}~\cite{Pit97,
Gnedin06, Gnedin97, GnedinPit} this class of exchangeable
random partitions can all be constructed by a random closed
set $Z \subset[0,1]$ where $Z$ is the scaled range of an
$\alpha-$stable subordinator conditioned on its value at a
fixed time. Conditional on $t$ this construction produces
the \PK$(\rho_{\alpha}|t)$ partition, where the conditional
EPPF was derived by Pitman~\cite{Pit02}.

Our work may be divided into two parts. One is the use of special
functions and the theory of fractional calculus to help interpret
these EPPF's and in many cases to obtain explicit numerical
calculations. The other is to use some interesting probability
distribution theory to obtain explicit results for large classes of
EPPF's and which in turn yields calculations for various special
functions.

Specifically we will show that for general $\alpha$ these EPPF's may
be represented in terms of Fox $H$-functions, and for the case where
$\alpha$ takes on rational values in terms of Meijer $G$-functions.
There are several significant implications of these representations.
One is that a large number of special functions commonly appearing
in, for instance, physics, probability, finance or fractional
calculus, can be represented in terms of $H$- and $G$-functions.
See, for instance,~\cite{Hilfer,
Kilbas04,Kirbook,Mainardi03,MLP,MPS,Prudnikov90}. In addition,
because these functions are well understood this offers additional
interpretability of the relevant EPPF's. That is one is not merely
applying a numerical calculation. The case of rational $\alpha$ is
particularly interesting as calculations involving general Meijer
$G$-functions are at the heart of mathematical computer packages
such as \textit{Mathematica} and \textit{Maple.} This literally
allows one to explicitly calculate thousands of EPPF's, while in the
present literature only a few cases of explicit EPPF's are known.
However, as of yet, while in the general case of $\alpha$ we can
express many EPPF's in terms of $H$-transforms there does not exist
general mathematical packages to compute them. This brings up our
other approach which is based on probability distribution theory
associated with beta, gamma and stable random variables relying in
large part on a recent work of James~\cite{JamesLamp}, see
also~(\cite{JamesMean}), and which further relies on some results in
Perman-Pitman-Yor~\cite{PY92,PPY92, PY97} and Pitman~\cite{Pit06},
and some lesser known results for $S_{\alpha}$. In this regard, the
fact that the relevant components in the EPPF satisfy a forward
recursion, which follows from a backward equation as seen in Gnedin
and Pitman~\cite{Gnedin06}, plays an important role. As will be
discussed in Section~\ref{sec:recursion}, the recursion shows that
one need only calculate the simplest components to calculate all
components via a recursion. Specifically, it suffices to compute the
probability of having one block in a partition of integers
$\{1,\ldots,n\},$ that is,
$$
\mathbb{P}(\{1,\ldots,n\})=V_{n,1}[1-\alpha]_{n-1},
$$
where $V_{n,1}$ is the quantity that needs to be computed, and the
other notation will be explained shortly. We should note that even
$V_{n,1}$ was thought not to be easily calculated, however we will
demonstrate that this can be done quite readily. We will also, in
many cases, be able to give explicit expressions for the more
complex quantities. This also sets up some interesting
relationships between various special functions, integral
transforms, and probabilities.

We provide a discussion and definition of Fox $H$- and
Meijer $G$-functions in the appendix, which is obtained
from various sources. We now proceed to address some
preliminaries and present a more specific outline.

\subsection{Preliminaries}
Again from Pitman~\cite{Pit02}, for $\Pi_{n}$ an exchangeable
partition of $\{1,2,\ldots,n\}$ and a particular partition
$\{A_{1},\ldots,A_{k}\}$ of $\{1,2,\ldots,n\}$ with
$|A_{i}|=n_{i}$ for $1\leq i\leq k,$ where $n_{i}\ge 1$ and
$\sum_{i=1}^{k}n_{i}=n,$ then letting,
$$
[a]_n = a (a+1) (a+2) \cdots (a+n-1) =
\frac{\Gamma(a+n)}{\Gamma(a)},
$$
the conditional EPPF, associated with the
\PK$(\rho_{\alpha}|t)$ partition, is defined as $
p_{\alpha}(n_{1},\ldots,n_{k}|t)=\mathbb{P}(\Pi_{n}=\{A_{1},\ldots,A_{k}\}|t),
$ where
\begin{equation}\label{PitEPPF}
p_{\alpha}(n_{1},\ldots,n_{k}|t)=
\mathbb{G}_{\alpha}^{(n,k)}(t)\prod_{j=1}^{k}[1-\alpha]_{n_{j}-1}
\end{equation}
and one can write
\begin{eqnarray}
\label{bigG}
\mathbb{G}_{\alpha}^{(n,k)}(t) &=&
\frac{\alpha^{k}}{t^{n}\Gamma(n-k\alpha)f_{\alpha}(t)}
\left[\int_{0}^{t}f_{\alpha}(t-v)v^{n-k\alpha-1}dv\right]
\\\nonumber
&=&\frac{\alpha^{k}t^{-k\alpha}}{\Gamma(n-k\alpha)f_{\alpha}(t)}
\left[\int_{0}^{1}f_{\alpha}(tu){(1-u)}^{n-k\alpha-1}du\right].
\end{eqnarray}
Using the terminology in~\cite{Gnedin06}, call the
\PK$(\rho_{\alpha}|t)$ partitions the
$(\alpha|t)$-partitions.

Now suppressing dependence on $\alpha$ and $\gamma,$ for
each $n$ and $k,$ set
$$
V_{n,k}=\int_{0}^{\infty}\mathbb{G}_{\alpha}^{(n,k)}(t)\gamma(dt).
$$
Pitman~\cite{Pit02} shows that the EPPF of the
\PK$(\rho_{\alpha},\gamma)$ partition is given by
\begin{equation}
\label{EPPFu}
p_{\alpha,\gamma}(n_{1},\ldots,n_{k})=V_{n,k}\prod_{j=1}^{k}[1-\alpha]_{n_{j}-1}.
\end{equation}
Note that by setting $\gamma$ to be point mass at $t,$~(\ref{EPPFu})
equates with~(\ref{PitEPPF}). The most well-known member of this
class is the case where for $\theta>-\alpha,$ $\gamma$ corresponds
to the distribution of the random variable $S_{\alpha,\theta}$
having density
\begin{equation}
f_{S_{\alpha,\theta}}(t)=\frac{\Gamma(\theta+1)}{\Gamma(\theta/\alpha+1)}t^{-\theta}f_{\alpha}(t)
\label{stabletemp}
\end{equation}
and satisfies for $\delta+\theta>-\alpha$,
$$
\mathbb{E}[S^{-\delta}_{\alpha,\theta}]=\frac{\Gamma(\theta+1)}{\Gamma(\theta/\alpha+1)}
\mathbb{E}[S^{-(\delta+\theta)}_{\alpha}]=\frac{\Gamma(\frac{(\theta+\delta)}{\alpha}+1)}
{\Gamma({\theta+\delta}+1)}\frac{\Gamma(\theta+1)}{\Gamma(\theta/\alpha+1)}.
$$
Note $S_{\alpha,0}\overset{d}=S_{\alpha}.$ Furthermore, the random
variables satisfy the remarkable identity that may be found in
Pitman~\cite[section 4.2]{Pit06} and Perman, Pitman and
Yor~\cite{PPY92}. That is, for any $\theta>-\alpha$,
\begin{equation}
\frac{1}{S_{\alpha,\theta}}\overset{d}=\frac{\beta_{\theta+\alpha,1-\alpha}}{S_{\alpha,\theta+\alpha}},
\label{key2}
\end{equation}
where $\beta_{\theta+\alpha,1-\alpha}$ is a
beta$(\theta+\alpha,1-\alpha)$ random variable independent of
$S_{\alpha,\theta}.$ Note that the identity~(\ref{key2}) plays an
important role in the recent work of~\cite{JamesLamp} and hence,
through that work, will play a prominent role here. That is,
using~(\ref{stabletemp}) as the mixing density gives the EPPF of
the two parameter $(\alpha,\theta)$ Poisson-Dirichlet
distribution, say, $\mbox{\scriptsize{PD}}(\alpha,\theta),$ given
by
\begin{equation}
p_{\alpha,\theta}(n_{1},\ldots,n_{k})=\frac{\prod_{l=1}^{k}(\theta+l\alpha)}
{[\theta+1]_{n-1}}\prod_{j=1}^{k}[1-\alpha]_{n_{j}-1}.
\label{PitmanEwens}
\end{equation}
The quantity in (\ref{PitmanEwens}) extends beyond the stable case,
as it is also defined for the cases of $\alpha=0$ and $\theta>0$
corresponding to the famous result related to the Dirichlet process,
otherwise known as the Chinese restaurant process, or Ewens
$(0,\theta)$-partitions. The other possibility is the case of
$-\infty\leq \alpha<0$ and $\theta=m|\alpha|,$ for $m=1,2\ldots,$
referred to as $(\alpha,|\alpha|m)$ partitions. The
\mbox{\scriptsize{PD}}$(\alpha,\theta)$ plays an important role in a
variety of diverse applications. See Pitman~\cite{Pit06} for a
general overview and set of references, and in particular, its
relation to Bessel and Brownian phenomena. See
Bertoin~\cite{BerFrag} for its role in terms of fragmentation and
coagulation phenomena. See Ishwaran and
James~\cite{IshwaranJames2001, IshwaranJames2003} and
Pitman~\cite{Pit96} for some applications in Bayesian statistics.

A remarkable fact, see Gnedin and Pitman~\cite[Theorem 12]{Gnedin06}
and Pitman~\cite[Theorem 4.6, p. 86]{Pit06}, is that the EPPF's
generated by (\ref{PitEPPF}), that is (\ref{EPPFu}), and mixtures of
the Ewens $(0,\theta)$-partitions and $(\alpha,|\alpha|m)$
partitions, constitute the only infinite EPPF's having Gibbs form,
that is, infinite EPPF's of the form
$$
c_{n,k}\prod_{j=1}^{k}w_{n_{j}}.
$$
This, as discussed in regards to the
\mbox{\scriptsize{PD}}$(\alpha,\theta)$ family, has potential
implications both from a practical and theoretical point of view in
a variety of disciplines. In particular, the $(\alpha|t)$-partitions
constitute the largest and most diverse of such classes. However,
there are only a few examples where $V_{n,k}$ has been computed.
Besides the \mbox{\scriptsize{PD}}$(\alpha,\theta)$ case, there are
also the models formed by taking $\gamma$ as a density proportional
to ${\mbox e}^{-bt}f_{\alpha}(t)$. In addition, Pitman~\cite[section
8]{Pit02} and Pitman~\cite[section 4.5, p.90]{Pit06} show that
conditional EPPF of the $(1/2|t^{-2})$-partition, corresponding to
the \emph{Brownian excursion partition}, is such that
$\mathbb{G}_{1/2}^{(n,k)}(t)$ can be expressed in terms of Hermite
functions~\cite[section 10.2]{Lebedev72}. This explicit result is
due in part to the fact that $S_{1/2}$ is equivalent in distribution
to an inverse gamma distribution with shape $1/2,$ and hence in
contrast to the case of general $S_{\alpha}$ has a simple explicit
density. However, clearly, given the fact that $\gamma$ may be quite
arbitrary, it is not enough to simply know the explicit form of the
density of $f_{\alpha}.$

\subsection{Goals and Outline}
Faced with this our goal becomes quite clear. Find methods
to explicitly calculate and hopefully provide further
interpretation of the quantities
$\mathbb{G}_{\alpha}^{(n,k)}(t)$ and $V_{n,k}$ for general
$\alpha.$ We wish to emphasize that we are not interested
in merely suggesting crude numerical methods which do not
have interpretability.

Our first task will be to provide an answer to a question
posed by Pitman~\cite[Problem 4.3.3,  p. 87]{Pit06}, which
goes beyond merely wanting a numerical calculation. We
paraphrase it as follows,

\begin{quote}
Pitman~\cite[section 4.5, eq. (4.59) and (4.67)]{Pit06}
shows that in the case of $\alpha=1/2$, the integral
$\mathbb{G}_{\alpha}^{(n,k)}(t)$ can be simply expressed in
terms of an entire function of a complex variable, the
\emph{Hermite function}, which has been extensively
studied. It is natural to ask whether
$\mathbb{G}_{\alpha}^{(n,k)}(t)$ might be similarly
represented in terms of some entire functions with a
parameter $\alpha,$ which reduces to the Hermite function
for $\alpha=1/2.$
\end{quote}

In sections~\ref{sec:condEPPF} and~\ref{sec:MeijerG} of
this paper we provide an answer to this question by showing
that $\mathbb{G}_{\alpha}^{(n,k)}(t)$ can be expressed as
the ratio of Fox $H$-functions in the case of general
$\alpha$ and Meijer $G$-functions in the case where
$\alpha$ is rational. In particular, in
section~\ref{sec:hermite}, we recover the case of the
Hermite function based on the calculus of Meijer
$G$-functions when $\alpha=1/2$, and show that for general
rational $\alpha$ these may be expressed as ratios of sums
of generalized hypergeometric functions. Additionally, in
those 2 sections we show that
$\mathbb{G}_{\alpha}^{(n,k)}(t)$ can be expressed in terms
of densities derived from $S_{\alpha,k\alpha}$ and
corresponding beta random variables. In
section~\ref{sec:uncondEPPF} we obtain results for the
unconditional Gibbs models, that is, calculations for
$V_{n,k}$. In section\ref{sec:Grational} we show that one
may use the calculus of Meijer $G$-functions to express
many $V_{n,k}$ in terms of $G$-functions which are then
readily computable. Sections~\ref{sec:ModBessel}
and~\ref{sec:genhyper} provide new specific examples of
EPPF's. Section~\ref{sec:distSX} represents our first real
departure from calculations of EPPF's based on the theory
of Meijer $G$-functions. In particular, we demonstrate that
one can obtain calculations for all values of $\alpha$
using random variables connected to
Lamperti~\cite{Lamperti}, which have been recently studied
in James~\cite{JamesLamp}. Interesting special examples,
which can be computed by various other means, are given in
sections~\ref{sec:2F1} and~\ref{sec:3F2}.
Section~\ref{sec:recursion}, albeit short, is a pivotal
section as it demonstrates the important role of recursion
formulae. This, as mentioned earlier, shows that one only
need to focus on calculation of the $V_{n,1}$ terms in
order to obtain the general $V_{n,k}.$ As a side note, we
believe that even in the case where $V_{n,k}$ may be
represented in terms of $G$-functions and therefore
computable, it might be more efficient to use the recursion
when one is interested in problems potentially involving
the calculation of all $V_{n,k}$ for $k=1,\ldots, n$ and
$n=1,2,\ldots.$ Such problems occur, for instance, in the
implementation of generalized Chinese restaurant schemes as
can be seen in Ishwaran and
James~\cite{IshwaranJames2003}~(see also
Pitman~\cite[section 3.1]{Pit06}. In those cases $n$
represents sample size of data and can be in the thousands.
In addition the recursion can be used to obtain new
recursive relationships for various quantities, we will
demonstrate that in section~\ref{sec:EPPFMar}. In
section~\ref{sec:EPPFprobtran}, we describe purely
probabilistic methods to calculate $V_{n,1}$ and in fact
obtain expression for general $V_{n,k}.$  We point out
again that there are very few explicit examples of EPPF's
so it is rather striking that we will now show how to
obtain many of them. Section~\ref{sec:Lamperti} introduces
what we call the \emph{Lamperti class} which can be seen as
an extension of section~\ref{sec:distSX}.
Section~\ref{sec:BetaGamma} introduces what we call the
\emph{Beta-Gamma class}. Remarkably this class shows that
$V_{n,k}$ may be expressed in terms of expectations
involving only beta and gamma random variables.
Section~\ref{sec:EPPFMar} represents a non-trivial
application of Proposition 7.4 which exploits the quite
broad results in Pitman and Yor~\cite{PY01}.

\begin{rem} Throughout we will write $G_{\delta}$ to represent a
gamma$(\delta,1)$ random variable and $\B_{a,b}$ to represent a
beta$(a,b)$ random variable. Furthermore, unless otherwise stated,
we will assume that when we write products of random variables, that
means the individual random variables are independent.
\end{rem}

\section{Conditional EPPF}\label{sec:condEPPF}
From Schneider~\cite{Schneider86}~(see also
\cite{Mainardi03, Metzler00}), one may represent the
density of $S_{\alpha}$ in terms of an $H$-function as
follows.

\begin{equation}\label{stableH}
f_{\alpha}(t) = \frac{1}{\alpha} H_{2,2}^{1,1} \left[ \; t \left|
\begin{matrix}
\left(1-\frac{1}{\alpha},\frac{1}{\alpha}\right),(0,1)\\ \\
(0,1),(0,1)
\end{matrix}\; \right.\right],\qquad t>0.
\end{equation}
Applying~(\ref{AppHlower2}) in the appendix, one can write
\begin{equation}\label{stablesimp}
f_{\alpha}(t) = \frac{1}{\alpha} H_{1,1}^{0,1} \left[ \; t \left|
\begin{matrix} (1-\frac{1}{\alpha},\frac{1}{\alpha})
\\ \\(0,1)  \end{matrix}\; \right.\right],
\qquad t>0,
\end{equation}
and use this to describe $\mathbb{G}_{\alpha}^{(n,k)}(t).$

\begin{thm}\label{thm1condEPPF}The function
$$
\mathbb{G}_{\alpha}^{(n,k)}(t)
$$
defined by~(\ref{bigG}) appearing in~(\ref{PitEPPF}) can be
expressed as follows.
\begin{enumerate}
\item[(i)]$\mathbb{G}^{(n,k)}_{\alpha}(t)$ is representable in
terms of ratios of probability densities as,
$$
\mathbb{G}^{(n,k)}_{\alpha}(t)=\frac{\alpha^{k-1}\Gamma(k)}
{\Gamma(n)}\frac{\tilde{f}_{\alpha,(n,k)}(t)}{f_{\alpha}(t)},
$$
where $\tilde{f}_{\alpha,(n,k)}$ denotes the density of the
random variables
\begin{equation}
\dfrac{S_{\alpha,k\alpha}}{\B_{k\alpha,n-k\alpha}}
\overset{d}=\frac{S_{\alpha,(k-1)\alpha}}{\B_{(k-1)\alpha+1,n-1-(k-1)\alpha}},
\label{equaldist}
\end{equation}
for $k=1,\ldots,n.$ In particular, for $k=1$,
\begin{equation}
\frac{S_{\alpha}}{\B_{1,n-1}}\overset{d}=\dfrac{S_{\alpha,\alpha}}{\B_{\alpha,n-\alpha}}.
\label{equaldist2}
\end{equation}
\item[(ii)]For all $0<\alpha<1,$ $\mathbb{G}^{(n,k)}_{\alpha}(t)$ is
expressible as the ratio of Fox $H$-functions,
$$
\mathbb{G}^{(n,k)}_{\alpha}(t)=\alpha^{k}
\frac{H_{1,1}^{0,1} \left[ \; t \left|
\begin{matrix} (1-\frac{1}{\alpha}-k,\frac{1}{\alpha})
\\ \\(-n,1)  \end{matrix}\; \right.\right]} {H_{1,1}^{0,1} \left[ \; t \left|
\begin{matrix} (1-\frac{1}{\alpha},\frac{1}{\alpha})
\\ \\(0,1)  \end{matrix}\; \right.\right]}.
$$
\end{enumerate}
\end{thm}

\begin{proof} Let us proceed by first deriving the density of
$S_{\alpha,k\alpha}/\B_{k\alpha,n-k\alpha}.$ First for
clarity we use the fact that $\B_{k\alpha,n-k\alpha}$ has
density,
$$
\frac{\Gamma(n)}{\Gamma(k\alpha)\Gamma(n-k\alpha)}u^{k\alpha-1}{(1-u)}^{n-k\alpha-1},
$$
for $0<u<1.$ Now setting $\theta=k\alpha$ in~(\ref{stabletemp}),
we see that the usual operations to obtain the density involves
$$
[u^{-k\alpha+1}t^{-k\alpha}]f_{\alpha}(tu)u^{k\alpha-1}{(1-u)}^{n-k\alpha-1}.
$$
Hence due to the above cancelations, the density is
$$
\tilde{f}_{\alpha,(n,k)}(t)=\frac{\Gamma(n)\alpha}{\Gamma(k)\Gamma(n-k\alpha)}
t^{-k\alpha}\int_{0}^{1}f_{\alpha}(tu){(1-u)}^{n-k\alpha-1}du.
$$
Now it follows that,
$$
\B_{k\alpha,n-k\alpha}\overset{d}=\B_{k\alpha+1-\alpha,n-k\alpha-1+\alpha}\B_{k\alpha,1-\alpha}
$$
Furthermore, setting $\theta=k\alpha$ in~(\ref{key2}), one
gets
$$
\frac{S_{\alpha,k\alpha}}{\B_{k\alpha,1-\alpha}}\overset{d}=S_{\alpha,k\alpha-\alpha}.
$$
These two points yield the equivalence in~(\ref{equaldist})
and~(\ref{equaldist2}).

In order to establish statement~(ii), first write
$$
\mathbb{H}^{(n,k)}_{\alpha}( t)=\frac{
(tu)^{-k\alpha}}{\Gamma(n-k\alpha)}
 \int_{0}^{1}f_{\alpha}(tu) u^{k\alpha} (1-u)^{n-k\alpha-1}du.
$$
Substituting the expression for $f_{\alpha}(ut)$
with~(\ref{stableH}), and then using~$(\ref{AppHshift})$ in
the appendix to obtain an expression for
$(tu)^{-k\alpha}f_{\alpha}(ut)$, one sees that
\begin{eqnarray}
\mathbb{H}^{(n,k)}_{\alpha}( t) &\equiv&
\frac{1}{\alpha\Gamma(n-k\alpha)} \int_{0}^{1} u^{k\alpha}
(1-u)^{n-k\alpha-1}\nonumber\\
&& \qquad \qquad\qquad\times H_{2,2}^{1,1} \left[ \; tu \left|
\begin{matrix} (1-\frac{1}{\alpha}-k,\frac{1}{\alpha}),
(-k\alpha,1)
\\ \\(-k\alpha,1), (-k\alpha,1)  \end{matrix}\; \right.\right]
du\nonumber\\
&=& \frac{1}{\alpha} H_{1,1}^{0,1} \left[ \; t \left|
\begin{matrix} (1-\frac{1}{\alpha}-k,\frac{1}{\alpha})
\\ \\(-n,1)  \end{matrix}\; \right.\right],\label{Htransform}
\end{eqnarray}
where the last equality follows from identity 2 in~\cite[p.
355]{Prudnikov90} and some manipulations based
on~(\ref{AppHzero1}),~(\ref{AppHlower1}) and~(\ref{AppHlower2}). The
result is concluded by applying the
representation~(\ref{stablesimp}).
\end{proof}

Statement~(i) of Theorem~\ref{thm1condEPPF} provides a
probabilistic interpretation of
$\mathbb{G}^{(n,k)}_{\alpha}(t)$. We shall see that, in
particular, the distributional identity~(\ref{equaldist2})
will play a crucial role in applying probabilistic
arguments to calculations based on
$\mathbb{G}^{(n,k)}_{\alpha}(t)$. Statement~(ii) expresses
$\mathbb{G}^{(n,k)}_{\alpha}(t)$ in terms of Fox
$H$-functions, which, among other things, allows one to
make precise interpretations of it. Overall,
Theorem~\ref{thm1condEPPF} sets up a myriad of duality
relationships between probabilistic quantities based on
stable and beta random variables and a very large class of
special functions. In particular, as we will show, one can
use statement~(i) to obtain explicit calculations for
various Fox $H$-functions which are not yet readily
computable by other means. Now when $\alpha=m/r,$ the next
result shows that expressions in statement~(ii) reduces to
ratios of Meijer $G$-functions. The significance being
there is that calculations of Meijer $G$-functions are
readily available in \textit{Mathematica} and other
mathematical softwares.

\section{Rational Values, Products of Beta and Gamma Random
variables and Meijer $G$-Functions}\label{sec:MeijerG} When
$\alpha=m/r$ for integers $1\leq m<r,$  the stable random
variable $S_{m/r}$ can be represented in terms of
independent beta and gamma random variables as follows.
$\left(\dfrac{m}{S_{\frac{m}{r}}}\right)^m \stackrel{d}{=}
r^r Z_{m,r}$, where
$$
Z_{m,r} \overset{d}= \left(\prod_{i=1}^{m-1}
\B_{\frac{i}{r},\frac{i}{m}-\frac{i}{r}}\right) \left(
\prod_{j=m}^{r-1} G_{\frac{j}{r}}\right).
$$
This result may be found in Chaumont and
Yor~\cite[p.113]{Chaumont}~(see also~\cite[section
6]{JamesLamp}). Now, by Theorem~9 in Springer and
Thompson~\cite[p.733]{Springer70}, the density of $Z_{m,r}$
can be represented in terms of Meijer $G$-functions as
follows,
\begin{equation}
f_{Z_{m,r}}(t) = K_{m/r} G^{r-1,0}_{m-1,r-1}\left(t
 \left|
\begin{matrix} (\frac{i}{m}-1)_{1}^{m-1}
\\ \\(\frac{j}{r}-1)_{1}^{r-1} \end{matrix} \right.
\right),
\end{equation}
where
\begin{equation}\label{K}
K_{m/r} = \prod_{i=1}^{m-1}
\frac{\Gamma\left(\frac{i}{m}\right)}{\Gamma\left(\frac{i}{r}\right)}
\prod_{j=m}^{r-1} \frac{1}{\Gamma\left(\frac{j}{r}\right)},
\end{equation}
Hence the density of $S_{\frac{m}{r}}$ in $s$ is given by
\begin{eqnarray}
&&f_{Z_{m,r}}\left(\frac{m^m}{r^r s^m}\right) \times
\left|-\frac{m^{m+1}}{r^r s^{m+1}}\right|\nonumber\\
&&\qquad= K_{m/r} r^{\frac{r}{m}}G_{r-1,m-1}^{0,r-1}
\left(\left(\dfrac{r^{r}}{m^{m}}\right)s^m
 \left|
\begin{matrix} (1-\frac{1}{m}-\frac{i}{r})_{1}^{r-1}
\\ \\(1-\frac{1}{m}-\frac{j}{m})_{1}^{m-1}\end{matrix} \right.
\right)\label{Smr},
\end{eqnarray}
by absorbing the term $\left(m^m r^{-r} s^{-m}\right)^{(m+1)/m}$
using~(\ref{AppHshift}) and then applying~(\ref{AppHanalytic}). Note
one could have used the result of \cite{Zolotarev94} to
obtain~(\ref{Smr}). However that result does not equate $S_{m/r}$
with beta and gamma random variables.

Now define the vectors,
\begin{equation}\label{Pup}
\Pup =
\left(\left(1-\frac{1}{m}-\frac{i+k}{r}\right)_{1}^{r-1},
\left(\frac{i-1}{m}-\frac{k}{r}\right)_{1}^{m}\right)
\end{equation}
and
\begin{equation}\label{Pdown}
\Pdown =
\left(\left(1-\frac{1}{m}-\frac{j}{m}-\frac{k}{r}\right)_{1}^{m-1},
\left(\frac{j-1-n}{m}\right)_{1}^{m}\right).
\end{equation}

\begin{thm}\label{thm2condEPPF}
Let $m,r$ denote integers such that $1\leq m<r$. Then, for
$\alpha=m/r$,~(\ref{bigG}) is expressible as
$$
\mathbb{G}^{(n,k)}_{\frac{m}{r}}(t)=
\frac{m^{k-n}G_{m+r-1,2m-1}^{0,m+r-1}
\left(\left(\dfrac{r^{r}}{m^{m}}\right)t^m
 \left|
\begin{matrix} \Pup
\\ \\ \Pdown \end{matrix} \right.
\right)}{G_{r-1,m-1}^{0,r-1}
\left(\left(\dfrac{r^{r}}{m^{m}}\right)t^m
 \left|
\begin{matrix} (1-\frac{1}{m}-\frac{i}{r})_{1}^{r-1}
\\ \\(1-\frac{1}{m}-\frac{j}{m})_{1}^{m-1} \end{matrix} \right.
\right)},
$$
or
$$
\mathbb{G}^{(n,k)}_{\frac{m}{r}}\left(\left(\frac{m}{r^{\frac{r}{m}}}\right)s^{\frac{1}{m}}\right)=
\frac{m^{k-n}G_{m+r-1,2m-1}^{0,m+r-1} \left(s
 \left|
\begin{matrix} \Pup
\\ \\ \Pdown \end{matrix} \right.
\right)}{G_{r-1,m-1}^{0,r-1} \left(s
 \left|
\begin{matrix} (1-\frac{1}{m}-\frac{i}{r})_{1}^{r-1}
\\ \\(1-\frac{1}{m}-\frac{j}{m})_{1}^{m-1} \end{matrix} \right.
\right)}.
$$
\end{thm}

\begin{proof}
From statement~(i) in Theorem~\ref{thm1condEPPF} it suffices to
obtain the density of $S_{m/r,km/r}\times \B_{km/r,n-km/r}.$ Using
the expression for $f_{m/r}$ in~(\ref{Smr}), we note that a solution
for the integral $ \int_{0}^{1}f_{m/r}(tu)(1-u)^{q-1}du$, for $q>0$,
is obtained from Corollary~\ref{AppGcor1} as
$$
\frac{K_{m/r} r^{\frac{r}{m}}\Gamma(q)}{m^q}
G_{m+r-1,2m-1}^{0,m+r-1} \left(\left(\frac{r^{r}}{m^{m}}\right)t^m
 \left|
\begin{matrix} (1-\frac{1}{m}-\frac{i}{r})_{1}^{r-1},
(\frac{i-1}{m})_{1}^{m}
\\ \\(1-\frac{1}{m}-\frac{j}{m})_{1}^{m-1},(\frac{j-1-q}{m})_{1}^{m}
\end{matrix} \right.
\right).
$$
Hence, absorbing the term $t^{-km/r}$ into the above expression
with $q=n-km/r$ by~(\ref{AppHshift}) and taking the ratio of the
resulting expression and~(\ref{Smr}) gives
$\mathbb{G}^{(n,k)}_{m/r}(t)$. The second result follows from
simple algebra.
\end{proof}

We next show how to recover the result of
Pitman~\cite[Corollary 4.11, p. 93]{Pit06} or~\cite{Pit02}
in the case of $m=1$ and $r=2$, and show that the other
rational cases of $\alpha$ may be expressed in terms of
generalized hypergeometric functions.

\subsection{Hermite and generalized hypergeometric functions}\label{sec:hermite}
\begin{enumerate}
\item[(i)] When $m=1$ and $r=2$, we see that Theorem~\ref{thm2condEPPF} recovers
almost immediately the result of Pitman~\cite[Corollary
4.11, p. 93]{Pit06} or~\cite{Pit02} as follows. One has
that $\mathbb{G}^{(n,k)}_{1/2}(t)$ is expressible as
$$
\frac{G^{0,2}_{2,1} \left(4t
 \left|
\begin{matrix} -\frac{1+k}{2},-\frac{k}{2}
\\ \\-n \end{matrix} \right.
\right)}{G^{0,1}_{1,0} \left(4t
 \left|
\begin{matrix} -\frac{1}{2}
\\ \\\overline{\hspace*{0.2in}} \end{matrix}
\right. \right)}=
\frac{(4t)^{-\frac{1+k}{2}-1}e^{-\frac{1}{4t}}
U\left(-\frac{k}{2}+n, \frac{3}{2}, \frac{1}{4t}\right)}
{(4t)^{-\frac{3}{2}}e^{-\frac{1}{4t}}},
$$
where $U(a,b,c)$ is the confluent hypergeometric function
of the second kind~(see~\cite[p. 263]{Lebedev72}). The
above ratio reduces to
$$
2^{-k+1} t^{-\frac{k}{2}+\frac{1}{2}}
U\left(-\frac{k}{2}-\frac{1}{2}+n, \frac{1}{2},
\frac{1}{4t}\right)
$$
via an application of the recurrence relation~\cite[p.
505]{Slater65}
$$
U(a,b,z) = z^{1-b} U(1+a-b,2-b,z).
$$
A change of variable $t = \frac{1}{2} \lambda^{-2}$ yields
the expression $2^{n-k} \lambda^{k-1} h_{k+1-2n}(\lambda)$
inside equation~(110) in~\cite{Pit02}, where
$h_{\nu}(\lambda)$ is the Hermite function of index
$\nu$~\cite[section 10.2]{Lebedev72}, based on the
following relationship,
$$
h_{\nu}(\lambda) = 2^{\nu/2} U\left(-\frac{\nu}{2},
\frac{1}{2}, \frac{\lambda^2}{2}\right).
$$
\item[(ii)]
Now, when $m=1$ and $r=3$, $\mathbb{G}^{(n,k)}_{1/3}(t)$ is
expressible as
$$
\frac{G^{0,3}_{3,1} \left(27t
 \left|
\begin{matrix} -\frac{2+k}{3},-\frac{1+k}{3}, -\frac{k}{3}
\\ \\-n \end{matrix} \right.
\right)}{G^{0,2}_{2,0} \left(27t
 \left|
\begin{matrix} -\frac{1}{3}, -\frac{2}{3}
\\ \\\overline{\hspace*{0.2in}} \end{matrix}
\right. \right)},
$$
where the $G$-functions at the numerator and at the
denominator are respectively
\begin{eqnarray*}
&&\frac{4\pi^2}{3^{\frac{k}{3}+4}t^{\frac{k}{3}+1}}\left[\frac{
\,_1{F}_2(1-n+\frac{k}{3};\frac{1}{3},\frac{2}{3};\frac{1}{27t})}
{\Gamma(n-\frac{k}{3})\Gamma(\frac{1}{3})\Gamma(\frac{2}{3})}
-\frac{
\,_1{F}_2(1-n+\frac{k}{3}+\frac{1}{3};\frac{2}{3},\frac{4}{3};\frac{1}{27t})}
{3t^{\frac{1}{3}}\Gamma(n-\frac{k+1}{3})\Gamma(\frac{2}{3})\Gamma(\frac{4}{3})}\right.\\
&&\left. \hspace*{2in}+\frac{
\,_1{F}_2(1-n+\frac{k}{3}+\frac{2}{3};\frac{4}{3},\frac{5}{3};\frac{1}{27t})}
{(3t^{\frac{1}{3}})^2
\Gamma(n-\frac{k+2}{3})\Gamma(\frac{4}{3})\Gamma(\frac{5}{3})}\right]
\end{eqnarray*} and
$$
\frac{2\pi}{3^{\frac{9}{2}}t^{\frac{4}{3}}}\left[
\frac{\,_0{F}_1(;\frac{2}{3};\frac{1}{27t})}{\Gamma(\frac{2}{3})}-
\frac{\,_0{F}_1(;\frac{4}{3};\frac{1}{27t})}{3t^{\frac{1}{3}}\Gamma(\frac{4}{3})}\right],$$
wherein, for non-negative integers $p$ and $q$, $p\leq q$ or
$p=q+1,|z|\leq 1$, and $b_j\neq 0,-1,-2,\ldots j=1,\ldots,q$,
\begin{equation}\label{pFq}
_pF_q(a_1,\ldots,a_p;b_1,\ldots,b_q;z) = \sum_{i=0}^\infty
\frac{[a_1]_i[a_2]_i\cdots[a_p]_i}{[b_1]_i[b_2]_i\cdots[b_q]_i}
\frac{z^i}{i!},
\end{equation}
is called the generalized hypergeometric
function~(see~\cite[Chapter IV]{Erdelyi53a},~\cite[p.
437]{Prudnikov90}), which is available in
\textit{Mathematica} as\break
\verb"HypergeometricPFQ[{a1,...,ap},{b1,...,bq},z]".
\end{enumerate}

\begin{rem} In general, when $\alpha=m/r$ with $r > 2$, both the numerator
and the denominator of the ratios in Theorem~\ref{thm2condEPPF}
are expressible in terms of sums of generalized hypergeometric
functions of the form,
$$
_{2m
-1}{F}_{m+r-2}(a_1,\ldots,a_{2m-1};b_1,\ldots,b_{m+r-2};\cdot) $$
and
$$
_{m-1}{F}_{r-2}(a_1,\ldots,a_{m-1};b_1,\ldots,b_{r-2};\cdot)
$$
respectively.
\end{rem}

\section{Unconditional Gibbs}\label{sec:uncondEPPF}
Hereafter we assume that the mixing distribution can be
represented as the density
$$
\gamma_{\alpha}(t)=h(t)f_{\alpha}(t)
$$
where $h(t)$ is a non-negative function such that
$$
\int_{0}^{\infty}h(t)f_{\alpha}(t)=\mathbb{E}[h(S_{\alpha})]=1.
$$
Note that mixing the density $\gamma_{\alpha}$ over
$\mathbb{G}^{(n,k)}_{\alpha}(t)$ leads to the following
class of operators
$$
\mathcal{I}^{(n,k)}_{\alpha}(h)=\int_{0}^{\infty}\mathbb{G}^{(n,k)}_{\alpha}(t)\gamma_{\alpha}(t)dt=\alpha^{k-1}\int_{0}^{\infty}h(t)
{H_{1,1}^{0,1} \left[ \; t \left|
\begin{matrix} (1-\frac{1}{\alpha}-k,\frac{1}{\alpha})
\\ \\ (-n,1)  \end{matrix}\; \right.\right]}dt
$$
which is a particular type of Fox $H$-integral transforms,
see, for instance,~\cite{Kirbook} and~\cite{Kilbas04}.

\begin{thm}Suppose that $h(t)$ is a non-negative integrable
function with respect to $f_{\alpha}.$ Then, without loss
of generality, set $\int_{0}^{\infty}h(t)f_{\alpha}(t)=1,$
and form the density
$\gamma_{\alpha}(t)=h(t)f_{\alpha}(t).$ Then the EPPF of
\PK$(\rho_{\alpha};\gamma_{\alpha})$ random partition has
Gibbs form
$$
p_{\alpha,\gamma_{\alpha}}(n_{1},\ldots,n_{k})=V_{n,k}\prod_{j=1}^{k}[1-\alpha]_{n_{j}-1}
$$
where, for $k=1,\ldots, n$,
\begin{eqnarray}\nonumber
V_{n,k}&=&\frac{\alpha^{k-1}\Gamma(k)}{\Gamma(n)}
\mathbb{E}\left[h\left(\frac{S_{\alpha,k\alpha}}{\B_{k\alpha,n-k\alpha}}\right)\right]\\\nonumber
&=&\frac{\alpha^{k-1}\Gamma(k)}{\Gamma(n)}
\mathbb{E}\left[h\left(\frac{S_{\alpha,(k-1)\alpha}}{\B_{(k-1)\alpha+1,n-1-(k-1)\alpha}}\right)\right]\\\nonumber
&=&\mathcal{I}^{(n,k)}_{\alpha}(h)
\end{eqnarray}
with $V_{1,1}=1.$
\end{thm}
This result follows directly from
Theorem~\ref{thm1condEPPF}.

\subsection{$\mathbf{G}$-transform for rational values}\label{sec:Grational} When $h(t)$ is set
to be a $G$-function, one can use, for instance,
Theorem~\ref{AppG2Gthm} in the appendix, to calculate many EPPF's as
follows.

\begin{prop}\label{thm5EPPFMN}
When $\alpha=m/r$ and $h(t)$ is expressible as
$$
C \times G^{u,v}_{w,x}\left(\sigma t
 \left|
\begin{matrix} (c_i)_{1}^{w}\\ \\(d_j)_{1}^{x} \end{matrix} \right. \right),
$$
where $C$ is a constant and notation in~(\ref{Gnotation2})
follows, then $V_{n,k}$ is
\begin{eqnarray*}
&&\frac{C K_{m/r} r^{r/m} m^{\rho+(x-w)-1+k-n}}{\sigma(2\pi)^{b^{\ast}(m-1)}} \\
&&\qquad \times G^{vm,(u+1)m+r-1}_{(x+1)m+r-1,(w+2)m-1} \left(
\frac{r^{r}}{(\sigma m)^{m}}
 \left|
\begin{matrix} \Pup,\Delta(m,-d_1),\ldots,
\Delta(m,-d_x)\\ \\\Pdown, \Delta(m,-c_1),\ldots,
\Delta(m,-c_w)\end{matrix} \right. \right),
\end{eqnarray*}
where $K_{m/r}$ is defined in~(\ref{K}), $\Delta(\ell,a)$, for any
integer $\ell$, is defined in~(\ref{Delta2}), and $\rho$ and
$b^{\ast}$ are defined for $G^{u,v}_{w,x}(\sigma t)$ according
to~(\ref{Gnotation2}).
\end{prop}
The result is just a specialization of
Theorem~\ref{AppG2Gthm} given in the appendix. We now
address two specific new examples.

\subsubsection{Example: Modified Bessel
functions}\label{sec:ModBessel}
\begin{prop} Suppose we take $\gamma_{1/2}(t)
\propto K_{\eta}(\sqrt{t}) f_{1/2}(t)$, or in other words, $h(t) =
K_{\eta}(\sqrt{t}) / \int K_{\eta}(\sqrt{t}) f_{1/2}(t) dt$, where
$K_{\eta}(\cdot)$ is the modified Bessel function of the third
kind or Macdonald function~\cite[Sec. 7.2.1]{Erdelyi53b}. Then,
$$
V_{n,k} = \frac{G^{0,4}_{4,1} \left(16
 \left|
\begin{matrix} -\frac{1+k}{2},-\frac{k}{2}, \frac{\eta}{2},
-\frac{\eta}{2}
\\\\ -n\end{matrix} \right.
\right)}{G_{3,0}^{0,3} \left(16
 \left|
\begin{matrix}-\frac{1}{2}, \frac{\eta}{2}, -\frac{\eta}{2}
\\ \\\overline{\hspace*{0.5in}}\end{matrix} \right.
\right)}.
$$
\end{prop}
\begin{proof} The result follows from Proposition~\ref{thm5EPPFMN}
by substituting $m=1$ and $r=1$ and recognizing
$$
h(t) = C \times K_{\eta}(\sqrt{t}) = C \times G^{2,0}_{0,2}
\left(\frac{t}{4}
 \left|
\begin{matrix} \overline{\hspace*{0.5in}}
\\\\ -\frac{\eta}{2}, \frac{\eta}{2}\end{matrix} \right.
\right)
$$
by~(\ref{AppBessel}), where
\begin{eqnarray*}
C^{-1} &=& \int K_{\eta}(\sqrt{t}) f_{1/2}(t) dt \\
&=& \frac{4}{\sqrt{\pi}} \int_0^\infty G^{2,0}_{0,2}
\left(\frac{t}{4}
 \left|
\begin{matrix} \overline{\hspace*{0.5in}}
\\ \\-\frac{\eta}{2}, \frac{\eta}{2}\end{matrix} \right.
\right) G_{1,0}^{0,1} \left(4t
 \left|
\begin{matrix} -\frac{1}{2}
\\ \\ \overline{\hspace*{0.5in}} \end{matrix} \right.
\right)dt\\
&=& \frac{16}{\sqrt{\pi}}G_{3,0}^{0,3} \left(16
 \left|
\begin{matrix}-\frac{1}{2}, \frac{\eta}{2}, -\frac{\eta}{2}
\\ \\\overline{\hspace*{0.5in}}\end{matrix} \right.
\right)
\end{eqnarray*}
due to Theorem~\ref{AppG2Gthm}.
\end{proof}

\subsubsection{Generalized Hypergeometric Functions as EPPF}\label{sec:genhyper} An
EPPF in terms of one generalized hypergeometric function defined
in~(\ref{pFq})~(see~\cite[Chapter IV]{Erdelyi53a}) results when
$V_{n,k} = \int_0^\infty \mathbb{G}^{(n,k)}_{\alpha}(t) \gamma(t)
dt$ is representable as some constant multiplies either
\begin{equation}\label{hyper1}
G^{1,p}_{p,q+1}(-z),\quad \mbox{for }p\leq q.
\end{equation}
or
\begin{equation}\label{hyper2}
G^{1,q+1}_{q+1,q+1}(-z),\quad \mbox{for } |z|\leq 1,
\end{equation}
due to~(\ref{AppGhyper}).

\begin{prop}
Suppose that $\alpha=1/r$ and $h(t) = C \times G^{x,1}_{w,x}
\left(\left(r^r\right)t
 \left|
\begin{matrix} (c_i)_{1}^{w}
\\ \\(d_j)_{1}^{x} \end{matrix} \right.
\right)$, where $C$ is a constant, notation in~(\ref{Gnotation2})
follows, and none of $(c_i)_{1}^{w}$ is identical to any of
$(d_j)_{1}^{x}$. When
\begin{itemize}
\item[(i)] $w \geq x+r$, or
\item[(ii)] $x=w-r+1$ and $w \geq r-1$,
\end{itemize}
$V_{n,k}$ is given by
$$
C^{\ast}_{w,x} \times
\,_{x+r}F_w\left(\left(\frac{i+k}{r}+\nu\right)_{1}^{r-1},\frac{k}{r}+\nu,
\left(d_i+\nu\right)_{1}^{x};(c_j+\nu)_{1}^{w};-1\right),
$$
where $\nu = 1-n$ and
$$
C^{\ast}_{w,x} = \frac{CK_{1/r}
r^r\prod_{i=1}^{r-1}\Gamma\left(\frac{i+k}{r }+\nu\right)
\Gamma\left(\frac{k}{r}+\nu\right) \prod_{i=1}^{x}\Gamma(d_i+\nu)}
{\prod_{j=1}^w\Gamma(c_j+\nu)},
$$
with $K_{1/r}$ defined in~(\ref{K}) with $m=1$, provided that $n
\neq d_j \neq \frac{k+\ell-1}{r}$ and $c_i \neq
\frac{k+\ell-1}{r}$ for $i=1,\ldots,w, j=1,\ldots,x,
\ell=1,\ldots,r$.
\end{prop}

\begin{proof} It follows from Proposition~\ref{thm5EPPFMN} with $m=1$ that $V_{n,k}$
is equal to $CK_{1/r}r^r$ multiplies
\begin{eqnarray*}
&&G_{x+r,w+1}^{1,x+r} \left(1
 \left|
\begin{matrix} (-\frac{i+k}{r})_{1}^{r-1},
-\frac{k}{r},(-d_i)_{1}^{x}
\\ \\-n, (-c_j)_{1}^{w}\end{matrix} \right. \right)\\
&&\quad=G_{x+r,w+1}^{1,x+r} \left(1
 \left|
\begin{matrix} \left(1-\left(\frac{i+k}{r}+\nu\right)\right)_{1}^{r-1},
1-\left(\frac{k}{r}+\nu\right),\left(1-\left(d_i+\nu\right)\right)_{1}^{x}
\\ \\0, (1-(c_j+\nu))_{1}^{w}\end{matrix} \right. \right),
\end{eqnarray*}
followed from~(\ref{AppHshift}). The last $G$-function takes
either the form of~(\ref{hyper1}) when $w \geq x+r$, or the form
of~(\ref{hyper2}) when $x=w-r+1$ and $w\geq r-1$, and, hence, the
result follows from~(\ref{AppGhyper}).
\end{proof}

\section{First distribution theory example: $S_{\alpha,\theta}$ given
$X_{\alpha,\theta}$}\label{sec:distSX} We now come to our first
result which does not rely on special properties of $G$- or
$H$-functions and importantly applies to all values of $\alpha.$ Let
$S_{\alpha}$ and $S_{\alpha,\theta}$ denote independent random
variables having laws described previously. Then, define the random
variables
\begin{equation}
X_{\alpha,\theta}=\frac{S_{\alpha}}{S_{\alpha,\theta}}
\label{ratio}
\end{equation}
whose laws have been recently studied in
James~\cite{JamesLamp}. They represent a natural
generalization of the random variable
$$
X_{\alpha}=\frac{S_{\alpha}}{S'_{\alpha}}
$$
where $X_{\alpha}\overset{d}=X_{\alpha,0}$ and
$S'_{\alpha}$ is independent of $S_{\alpha}$ and has the
same distribution. Remarkably although $S_{\alpha}$ does
not have a simple density, except for $\alpha=1/2$,
Lamperti~\cite{Lamperti}~(see also Chaumont and
Yor~\cite[exercise 4.2.1]{Chaumont}) shows that the density
of $X_{\alpha}$ is
\begin{equation}
f_{X_{\alpha}}(y)=\frac{\sin(\pi
\alpha)}{\pi}\frac{y^{\alpha-1}}{y^{2\alpha}+2y^{\alpha}\cos(\pi
\alpha)+1},\qquad\mbox{ for }y>0
\label{denX},\end{equation} with cdf
$$
F_{X_{\alpha}}(y)=1-\frac{1}{\pi\alpha}\cot^{-1}\left(\cot(\pi
\alpha)+\frac{y^{\alpha}}{\sin(\pi \alpha)}\right).$$ Note
furthermore, when $\alpha=1/2,$
\begin{equation} X_{1/2}\overset{d}=\frac{G'_{1/2}}{G_{1/2}}{\mbox
{ and
}}X_{1/2,\theta}\overset{d}=\frac{G_{\theta+1/2}}{G_{1/2}},
\label{Brownian}
\end{equation}
where $G'_{1/2}\overset{d}=G_{1/2}$ and all the gamma random
variables are independent. See~\cite[section 4.2]{JamesMean} for
more on the variables~(\ref{Brownian}).

Here we investigate the mixing distribution $\gamma$
corresponding to the law of $S_{\alpha,\theta}$ given
$X_{\alpha,\theta}=1.$ That is, the random variable with
density
\begin{equation}
\gamma_{\alpha}(t)=C_{\alpha,\theta}t^{1-\theta}f_{\alpha}(t)f_{\alpha}(t)
\label{Lam}
\end{equation}
where
$$
C_{\alpha,\theta}=1/f_{X_{\alpha,\theta}}(1).
$$
So from~(\ref{denX}),
$$
C_{\alpha,0}=\frac{2\pi(1+\cos(\pi\alpha))}{\sin(\pi\alpha)}.
$$
Now define, for $\theta>0,$
$$
\Delta_{\theta}(x|F_{X_{\alpha}})=\frac{1}{\pi}\frac{\sin(\pi \theta
F_{X_{\alpha}}(x))}{{[x^{2\alpha}+2x^{\alpha}\cos(\alpha
\pi)+1]}^{\frac{\theta}{2\alpha}}}.
$$
We now give a brief description of the density of
$X_{\alpha,\theta}$ for $\theta>0$ and general $\alpha$
which is due to James~\cite[Theorem 3.1]{JamesLamp}. When
$\theta=1,$
$$
f_{X_{\alpha,1}}(y)=\Delta_{1}(y|F_{X_{\alpha}})=\frac{1}{\pi}\frac{\sin(\pi
F_{X_{\alpha}}(y))}{{[y^{2\alpha}+2y^{\alpha}\cos(\alpha
\pi)+1]}^{\frac{1}{2\alpha}}}.
$$
Hence, the normalizing constant in this case satisfies
$$
1/C_{\alpha,1}=f_{X_{\alpha,1}}(1)=\frac{1}{\pi}\frac{\sin(\pi
F_{X_{\alpha}}(1))}{{[2(1+\cos(\alpha \pi))]}^{\frac{1}{2\alpha}}}
$$

In general, for $\theta>0,$
\begin{equation}
f_{X_{\alpha,\theta}}(y)=\int_{0}^{y}{(y-x)}^{\theta-1}\Delta'_{\theta}(x)dx
\label{Jamesden}
\end{equation} where, suppressing dependence on $F_{X_{\alpha}},$ $\Delta'$  denotes the derivative of $\Delta.$
Furthermore, define
\begin{equation}
\label{sumden} \vartheta^{(n,k)}_{j}(x)=\frac{{[\sin(\pi \alpha
)]}^{k-j}}{\pi}\frac{( x^{\alpha }\cos ( \alpha \pi ) +1)
^{j}x^{\alpha ( k-j) }}{\left[ x^{2\alpha }+2x^{\alpha }\cos (
\alpha \pi ) +1\right] ^{k}}.
\end{equation}
Before we proceed with the description of the EPPF, we provide a
representation of $\Delta_{k\alpha}$ connected with the random
variable $X_{\alpha,k\alpha}.$

\begin{lem}\label{lemDelta}For $0<\alpha<1,$ and $k=1,2,\ldots,$
$$
\Delta_{k\alpha}(x|F_{X_{\alpha}})=\sum_{j=0}^{k}\binom{k}{j}\sin\left(
\frac{\pi}{2}( k-j)\right) \vartheta^{(n,k)}_{j}(x)
$$
\end{lem}

\begin{proof}Since $k=1,2,\ldots,$ we first apply the multiple angle formula to
$$
\sin(\pi k\alpha F_{X_{\alpha}}(x)).
$$
The result is concluded by noting the following identities
which are given in James~\cite[Proposition 2.1]{JamesLamp},
\begin{eqnarray*}
\sin ( \pi \alpha F_{X_{\alpha }}( x) )  =\frac{%
x^{\alpha }\sin ( \alpha \pi ) }{\left[ x^{2\alpha
}+2x^{\alpha }\cos ( \alpha \pi ) +1\right] ^{1/2}}
\end{eqnarray*} and
\begin{eqnarray*}
\cos ( \pi \alpha F_{X_{\alpha }}( x) )  =\frac{ x^{\alpha }\cos (
\alpha \pi ) +1}{\left[ x^{2\alpha }+2x^{\alpha }\cos ( \alpha \pi )
+1\right] ^{1/2}}.
\end{eqnarray*}
\end{proof}

Then we have the general description of the EPPF.

\begin{prop}Suppose $\gamma_{\alpha}$ is specified
as~(\ref{Lam}) for $\theta>-\alpha.$ Then, for $n=2,3\ldots,$ and
$k=1,\ldots,n,$
$$
V_{n,k}=C_{\alpha,\theta}\frac{\alpha^{k-1}\Gamma(\theta/\alpha+k)}{\Gamma(\theta+n-1)}\int_{0}^{1}(
1-x) ^{n+\theta -2}\Delta_{\theta+k\alpha}(
x|F_{X_{\alpha}}) dx
$$
\begin{enumerate} \item[(i)] In particular, when $\theta=0$, we
obtain
\begin{equation}
\label{secondV} V_{n,k}= \frac{\alpha^{k-1}\Gamma(k)}{\Gamma(n-1)}
\sum_{j=0}^{k}\binom{k}{j}\sin \left( \frac{\pi}{2}( k-j)\right)
\varphi^{(n,k)}_{j}
\end{equation}
where
$$
\varphi^{(n,k)}_{j}=\frac{2\pi(1+\cos(\alpha\pi))}
{\sin(\pi\alpha)}\int_{0}^{1}{(1-x)}^{n-2}\vartheta^{(n,k)}_{j}(x)dx.
$$
Note that $\vartheta^{(n,k)}_{j}(x)>0$ for $0<x<1.$
\item[(ii)] When $\theta=0$ and $n>1,$
\begin{equation}
\label{firstV}
V_{n,1}=\frac{2(1+\cos(\alpha\pi))}{\Gamma(n-1)}\int_{0}^{1}\frac{{(1-x)}^{n-2}x^{\alpha}}{\left[
x^{2\alpha }+2x^{\alpha }\cos ( \alpha \pi ) +1\right]}dx.
\end{equation}
\end{enumerate}
\end{prop}

\begin{proof} It is easy enough to work directly with the
expression
$$
\mathbb{G}^{(n,k)}_{\alpha}(t)=\frac{\alpha^{k}t^{-k\alpha}}{\Gamma(n-k\alpha)f_{\alpha}(t)}
\left[\int_{0}^{1}f_{\alpha}(tu){(1-u)}^{n-k\alpha-1}du\right].
$$
Mixing relative to $\gamma_{\alpha}(t)$ defined
in~(\ref{Lam}), one sees that
$$
\int_{0}^{\infty}t^{1-(\theta+k\alpha)}f_{\alpha}(tu)f_{\alpha}(t)dt=\frac{\Gamma(\theta/\alpha+k+1)}{\Gamma(\theta+k\alpha+1)}
f_{X_{\alpha,\theta+k\alpha}}(u).
$$
Now noting the form of the density of $X_{\alpha,\theta+k\alpha}$
from~(\ref{Jamesden}) and integrating with respect to $u$ lead to
the evaluation of the integral
$$
\int_{x}^{1}(u-x)^{\theta+k\alpha-1}{(1-u)}^{n-k\alpha-1}du=
{(1-x)}^{n+\theta-1}\frac{\Gamma(\theta+k\alpha)\Gamma(n-k\alpha)}{\Gamma(\theta+n)}.
$$
Ignoring constants for a moment this leads to an integral
of the form
$$
\int_{0}^{1}{(1-x)}^{n+\theta-1}\Delta'_{\theta+k\alpha}(x)dx
$$
Now, since $\theta>-\alpha$, it follows that $n+\theta>1$ when
$n=2,3,\ldots.$ This allows us to use integration by parts to get
$$
\int_{0}^{1}{(1-x)}^{n+\theta-1}\Delta'_{\theta+k\alpha}(x)dx=(n+\theta-1)\int_{0}^{1}{(1-x)}^{n+\theta-2}\Delta_{\theta+k\alpha}(
x|F_{X_{\alpha}}) dx
$$
which yields the general expression for $\theta>-\alpha$.
Statements~(i) and~(ii) then follow from an application of
Lemma~\ref{lemDelta}, and also the use of~(\ref{denX}).
\end{proof}
When $\alpha=m/r$, we may easily express $V_{n,k}$ in terms
of $G$-functions, which we leave to the reader. We focus on
two interesting cases.

\subsection{$_2F_1$ EPPF}\label{sec:2F1} When $\alpha=1/2$, we have that for
$\theta>-1/2$,
$$
C_{1/2,\theta}=\frac{2^{1-\theta}\pi}{\Gamma(\theta+1)}.
$$
One can then show that
$$
V_{n,k}=\frac{2^{\theta+1}}{\Gamma(\theta+1)}G^{1,2}_{2,2} \left(1
 \left|
\begin{matrix} -\theta-\frac{k}{2}, -\theta-\frac{k}{2}+\frac{1}{2}\\ \\
0,-\theta-n+\frac{1}{2}\end{matrix} \right. \right).
$$
Now applying ~(\ref{AppGhyper}) or identity~(2.9.15)
in~\cite{Kilbas04}, the last $G$-function reduces to
$$
\frac{\Gamma(\theta+\frac{k}{2}+1)\Gamma(\theta+\frac{k}{2}+\frac{1}{2})
} {\Gamma( \theta+n+\frac{1}{2})} \,_2F_1 \left(
\theta+\frac{k}{2}+1,\theta+\frac{k}{2}+\frac{1}{2};\theta+n+\frac{1}{2}
;-1\right),
$$
where
$$
    _2F_1\left(
    a,b;c
;x\right)=\sum_{i=0}^\infty \frac{[a]_i[b]_i}{[c]_i}
\frac{z^i}{i!}=\frac{\Gamma(c)}{\Gamma(b)\Gamma(c-b)}
\int_0^1\frac{
r^{b-1}(1-r)^{c-b-1}} {\left(1-xr\right)^{a}}dr
$$
is the Gauss hypergeometric function~(see \cite{Oberhettinger65}),
which is a special case of~(\ref{pFq}). Note this result can also
be checked by using the explicit density of $X_{1/2,\theta}.$

\subsection{$_3F_2$ EPPF}\label{sec:3F2} When $\alpha=1/3$, that is, $m=1$ and
$r=3$, one gets
$$
V_{n,k}=\frac{C_{1/3,\theta}3^{3\theta}}{[\Gamma(1/3)\Gamma(2/3)]^2}G^{2,3}_{3,3}
\left(1
 \left|
\begin{matrix} -\frac{k}{3}-\frac{2}{3}, -\frac{k}{3}-\frac{1}{3},
-\frac{k}{3}\\ \\
\theta-\frac{2}{3}, \theta-\frac{1}{3},-n\end{matrix} \right.
\right).
$$
Furthermore, the $G$-function reduces to
$$
\dfrac{\Gamma(\theta+\frac{k}{3}+1) \prod_{i=1}^4
\Gamma(\theta+\frac{k}{3}+\frac{i}{3})}
{\Gamma(\frac{1}{3})\Gamma(\frac{2}{3})\Gamma(n+\theta+\frac{1}{3})
\Gamma(2(\theta+\frac{k+3}{3}))} \,_3F_2\left(\begin{array}{c}
    \theta+\frac{k+2}{3},n-\frac{k}{3},\theta+\frac{k+3}{3}\\
    \\
    n+\theta+\frac{1}{3},2(\theta+\frac{k+3}{3}) \\
\end{array}
;1\right).
$$

\section{A recursive method for calculating EPPF's and Fox
$H$-transforms}\label{sec:recursion}

Definition~3 or equation~(8) of Gnedin and
Pitman~\cite{Gnedin06} establishes the following backward
recursion for all $V_{n,k}$, $n=1,2,\ldots;k=1,2,\ldots,n$,
\begin{equation}\label{backward}
V_{n,k} = (n-k\alpha) V_{n+1,k} + V_{n+1, k+1}
\end{equation}
with $V_{1,1}=1$.

One can turn~(\ref{backward}) around to obtain the following
forward recursion,
\begin{equation}
V_{n+1, k+1} = V_{n, k} - (n-k\alpha) V_{n+1, k}. \label{forward}
\end{equation}

The key point about the recursion~(\ref{forward}) is that
it enables computations of all $V_{n,k}$ for
$n=1,2,\ldots;k=1,2,\ldots,n$, provided that
$$
V_{n,1}=\frac{1}{\Gamma(n)}
\mathbb{E}\left[h\left(\frac{S_{\alpha}}{\B_{1,n-1}}\right)\right],
$$
for $n=1,2,\ldots$, are known. For clarity, we demonstrate this
point by considering the case of $n=3.$ When $V_{1,1},V_{2,1}$ and
$V_{3,1}$ are given, one can obtain all $V_{n,k}$ for
$n=1,2,3,k=1,\ldots,n$ as follows.
\begin{enumerate}
    \item[(i)] compute $V_{2,2} = V_{1,1} - (1-\alpha) V_{2,1}$;
    \item[(ii)] compute $V_{3,2} = V_{2,1} - (2-\alpha) V_{2,2}$;
    and
    \item[(iii)] compute $V_{3,3} = V_{2,2} - (2-2\alpha)
    V_{3,2}$.
\end{enumerate}

So quite fortunately we can focus on the relatively simpler task of
calculating $V_{n,1}.$ The simplicity occurs because the quantities
only depend on the distribution of the independent pairs
$(S_{\alpha},\B_{1,n-1})$, or equivalently,
$(S_{\alpha,\alpha},\B_{\alpha,n-\alpha}).$ In particular, when
$S_{\alpha}$ does not depend on $k$ or $n,$ we show some interesting
applications of the recursion in section~\ref{sec:EPPFMar}. One may
also see how this applies to the explicit expressions
in~(\ref{firstV}) and~(\ref{secondV}).

\subsection{An all purpose solution?}\label{sec:all}
It is widely believed that beyond the infinite series representation
of $f_{\alpha}$, there are no general explicit representations of
$f_{\alpha}.$ In fact this is false as one may use the
representation of Kanter~\cite{Kanter}, which we now describe.
Setting
$$
\mbox{K}_{\alpha}(u)={\left[\frac{\sin(\pi\alpha
u)}{\sin(\pi
u)}\right]}^{-\frac{1}{1-\alpha}}{\left[\frac{\sin((1-\alpha)
\pi u)}{\sin(\pi \alpha u)}\right]},
$$
it follows from Kanter~\cite{Kanter}~(see also
Devroye~\cite{DevroyeOne} and Zolotarev~\cite{Zolotarev86}) that
\begin{equation}
f_{\alpha}(s)=\frac{\alpha}{1-\alpha}s^{-1/(1-\alpha)}\int_{0}^{1}{\mbox
e}^{-{s}^{-\frac{\alpha}{1-\alpha}}\mbox{\scriptsize
K}_{\alpha}(u)}\mbox{K}_{\alpha}(u)du. \label{kan}
\end{equation}
That is,
\begin{equation}
S_{\alpha}\overset{d}={(\mbox{K}_{\alpha}(U)/G_{1})}^{(1-\alpha)/\alpha},
\label{StableSim}
\end{equation}
where $U$ is a uniform random variable on $[0,1]$, independent of a
standard exponential variable $G_{1}.$ So, using~(\ref{StableSim}),
one can obtain calculations for $V_{n,1}$ and hence $V_{n,k}$ by
simulating $U,G_{1}$ and a $\B_{1,n-1}$ variable. In addition, one
can certainly obtain new integral representations for $V_{n,1}$ and
$V_{n,k}.$

However, while this is certainly true in principle, and useful in
some cases, there is still quite a bit lacking in this
representation from an analytic point of view. For instance, it is
not obvious how to use~(\ref{kan}) to obtain the Laplace transform
of $S_{\alpha}$ or to obtain the nice form of the density of
$X_{\alpha}.$ In addition, one would require that $h(t)$ has a
manageable form, which is not the case for instance in
section~\ref{sec:EPPFMar}. Recall that even the known tractable case
of $\alpha=1/2$ does not always immediately lead to nice expressions
for $V_{n,1}.$ We will see in the next sections that one can do
quite nicely without resorting to~(\ref{kan}).

\section{Classes of EPPF's via Probability
Transforms}\label{sec:EPPFprobtran} The first sections focused on
the fact that we could use the fact that Mejier $G$-transforms can
now be easily computed to calculate $V_{n,k}$ and related
quantities. We also noted that while in theory the $V_{n,k}$ can be
calculated based on $H$-transforms, the practical tools to do this
are not yet available as not all $H$-functions can be computed
easily without resorting to numerical integrations. We now
demonstrate through using probability distribution theory how to
calculate $V_{n,1}$ and hence by the recursion all $(V_{n,k}).$ This
in turn will provide new methods to calculate many $H$-transforms as
well as $G$-transforms. We will focus on three general classes, but
certainly more can be constructed.

\subsection{Lamperti Class}\label{sec:Lamperti}

Here, we introduce an entire class of EPPF's based on
$X_{\alpha}$ and $X_{\alpha,k\alpha}.$ First, for a
positive integrable function $g,$ let
$$
1/L_{\alpha}=\frac{\sin(\pi
\alpha)}{\pi}\int_{0}^{\infty}\frac{g(y)y^{\alpha-1}dy}{y^{2\alpha}+2y^{\alpha}\cos(\pi
\alpha)+1}=\mathbb{E}[g(X_{\alpha})].
$$

\begin{prop}\label{thmLam1}
Suppose that for each $0<\alpha<1,$
$h(t)=L_{\alpha}\mathbb{E}[g(S'_{\alpha}/t)].$ That is,
$$
\gamma_{\alpha}(t)=L_{\alpha}\mathbb{E}[g(S'_{\alpha}/t)]f_{\alpha}(t).
$$
Then, \PK$(\rho_{\alpha},\gamma_{\alpha})$ has the following
properties.
\begin{enumerate}
\item[(i)]
$V_{n,1}=\dfrac{1}{\Gamma(n)}L_{\alpha}\mathbb{E}[g(X_{\alpha}\B_{1,n-1})]$.
\item[(ii)] For $k=1,2,\ldots,n$,
$$
V_{n,k}=\frac{\alpha^{k-1}\Gamma(k)}{\Gamma(n)}L_{\alpha}
\mathbb{E}[g(X_{\alpha,k\alpha}\B_{k\alpha,n-k\alpha})].
$$
\item[(iii)] Equivalently,
$$
V_{n,k}=\frac{\alpha^{k-1}\Gamma(k)}{\Gamma(n)}L_{\alpha}\mathbb{E}\left[
g\left(\B_{k,n-k}\sum_{j=1}^{k}X^{(j)}_{\alpha}D_{j}\right)\right]
$$
where $(X^{(j)}_{\alpha})$ are iid random variables equal in
distribution to $X_{\alpha}$ and, independently, $(D_{1},\ldots
D_{k})$ is a $k$-variate Dirichlet$(1,\ldots,1)$ random vector.
\end{enumerate}
\end{prop}

\begin{proof} Statements~(i) and~(ii) follow from
Theorem~\ref{thm1condEPPF} and the definition,
$$
X_{\alpha,k\alpha}=\frac{S'_{\alpha}}{S_{\alpha,k\alpha}}.
$$
Statement~(iii) amounts to showing that $$
X_{\alpha,k\alpha}\B_{k\alpha,n-k\alpha}\overset{d}=
\B_{k,n-k}\sum_{j=1}^{k}X^{(j)}_{\alpha}D_{j}.
$$
In order to do this, we refer back to the recent work of
James~\cite{JamesLamp,JamesMean}. First, write
$$
\B_{k\alpha,n-k\alpha}\overset{d}=\B_{k,n-k}\B_{k\alpha,k(1-\alpha)}.
$$
Then using the notation and the result in~\cite[Proposition
3.1(ii)]{JamesLamp},
$X_{\alpha,k\alpha}\overset{d}=M_{k\alpha}(F_{X_{\alpha}}).$ That
is, $X_{\alpha,k\alpha}$ is a Dirichlet mean random variable of
order $k\alpha$ indexed by $F_{X_{\alpha}}.$ Let $Y_{\alpha}$ denote
a Bernoulli random variable with success probability $\alpha,$
independent of $X_{\alpha}.$ Now, applying Theorem 2.1 in
James~\cite{JamesMean}, it follows that
$$
\B_{k\alpha,k(1-\alpha)}M_{k\alpha}(F_{X_{\alpha}})
\overset{d}=M_{k}(F_{X_{\alpha}Y_{\alpha}})
$$
where $M_{k}(F_{X_{\alpha}Y_{\alpha}})$ denotes a Dirichlet mean of
order $k$ indexed by the cdf of $X_{\alpha}Y_{\alpha},$
$F_{X_{\alpha}Y_{\alpha}}.$ Again, these specific random variables
are found in~\cite{JamesLamp}. Now an application of Proposition 9
in Hjort and Ongaro~\cite{Hjort}~(see also Proposition 4.5
in~\cite{JamesLamp}) yields
$$
M_{k}(F_{X_{\alpha}Y_{\alpha}})\overset{d}=\sum_{j=1}^{k}D_{j}M^{(j)}_{1}(F_{X_{\alpha}Y_{\alpha}})
$$
where
$M^{(j)}_{1}(F_{X_{\alpha}Y_{\alpha}})\overset{d}=M_{1}(F_{X_{\alpha}Y_{\alpha}})$
are independent. Now from Proposition 3.1(iii) in
James~\cite{JamesLamp}, one has
$$
X_{\alpha}\overset{d}=M_{1}(F_{X_{\alpha}Y_{\alpha}}),
$$
which completes the result.
\end{proof}

In view of this result, we proceed to give an expression for the
density of $X_{\alpha ,k\alpha }\B_{k\alpha ,n-k\alpha }.$ Define
$$
\psi^{(n,k)}_{j}(w)=(n-1)\int_{0}^{1/w}{(1-wx)}^{n-2}\vartheta^{(n,k)}_{j}(x)dx
$$
where $\vartheta^{(n,k)}_{j}(x)$ is given in~(\ref{sumden}).

\begin{lem}\label{lemLam}
For $n=2,3,\ldots$ and $k=1,2,\ldots,$  the density of $X_{\alpha
,k\alpha }\B_{k\alpha ,n-k\alpha }$ in $w$ is given by
$$
(n-1)\int_{0}^{1/w}{(1-wt)}^{n-2}\Delta_{k\alpha}(t|F_{X_{\alpha,k\alpha}})dt
$$
which, from Lemma~\ref{lemDelta}, can also be expressed
as
$$
\sum_{j=0}^{k}\binom{k}{j}\sin \left( \frac{\pi}{2}( k-j)\right)
\psi^{(n,k)}_{j}(w).
$$
\end{lem}

\begin{proof}
Using standard arguments one can write the density of $X_{\alpha
,k\alpha }\B_{k\alpha ,n-k\alpha }$  as
\begin{equation}
\label{one}
\frac{\Gamma(n)w^{k\alpha-1}}{\Gamma(k\alpha)\Gamma(n-k\alpha)}
\int_{w}^{\infty}(1-w/y)^{n-k\alpha-1}y^{-k\alpha}f_{X_{\alpha,k\alpha}}(y)dy.
\end{equation}
But it is not difficult to show that for any $\theta>-\alpha$,
$$
f_{1/X_{\alpha,\theta}}(y)=y^{-\theta}f_{X_{\alpha,\theta}}(y).
$$
Hence making the change $r=1/y,$ the expression in~(\ref{one}) is a
constant multiplies
$$
w^{k\alpha-1}\int_{0}^{1/w}(1-wr)^{n-k\alpha-1}f_{X_{\alpha,k\alpha}}(r)dr.
$$
Now arguing as in the proof of Lemma~\ref{lemDelta}, based on the
definition of the density of $X_{\alpha,k\alpha}$ provided
by~(\ref{Jamesden}), we will arrive at the integral
$$
w^{-1}\int_{tw}^{1}(1-u)^{n-k\alpha-1}{(u-tw)}^{k\alpha-1}du
$$
where $t$ refers to the argument in $\Delta'_{k\alpha}(t).$ So, it
follows that the density in~(\ref{one}) can be expressed as
$$
\frac{1}{w} \int_{0}^{\frac{1}{w}}( 1-wt) ^{n-1}\Delta'_{k\alpha}(
t) dt,
$$
and the result follows.
\end{proof}

This leads to another description of the quantities in
Proposition~\ref{thmLam1}.

\begin{prop}\label{thmLam2} Suppose that for each $0<\alpha<1,$
$$
\gamma_{\alpha}(t)=L_{\alpha}\mathbb{E}[g(S'_{\alpha}/t)]f_{\alpha}(t).
$$
Then, the EPPF for \PK$(\rho_{\alpha},\gamma_{\alpha})$ is such
that,
\begin{equation}
V_{n,k}= \frac{\alpha^{k-2}\Gamma(k)}{\Gamma(n)}
L_{\alpha}\sum_{j=0}^{k}\binom{k}{j}\sin \left( \frac{\pi}{2}(
k-j)\right) \zeta^{(n,k)}_{\alpha,j}
\end{equation}
where
$$
\zeta^{(n,k)}_{\alpha,j}=\frac{{[\sin(\pi \alpha
)]}^{k-j}}{\pi}\int_{0}^{\infty }\mathbb{E}\left[
g\left( \frac{\B_{1,n-1}}{r^{1/\alpha}}\right) \right] \frac{\left( r\cos \left( \pi \alpha \right) +1\right) ^{j}r^{\left( k-j\right) -1}%
}{\left[ r^{2}+2r\cos \left( \alpha \pi \right) +1\right] ^{k}}dr.
$$
\end{prop}

\begin{proof}This follows from Lemma~\ref{lemLam} and
Proposition~\ref{thmLam1} and applying a change of variable.
\end{proof}

\subsubsection{Example: Mittag Leffler and generalized Mittag
Leffler functions} This example is also influenced by some results
in~\cite{JamesLamp} and concerns the Mittag Leffler function and
some of its generalizations. Recall that the Mittag Leffler function
may be defined as
$$\textsc{E}_{\alpha,1}(-\lambda)=\mathbb{E}[{\mbox
e}^{-\lambda
S^{-\alpha}_{\alpha}}]=\sum_{l=0}^{\infty}\frac{{(-\lambda)}^{l}}{\Gamma(1+l\alpha)}=\mathbb{E}[{\mbox
e}^{-\lambda^{1/\alpha}X_{\alpha}}].
$$
Furthermore, there is an integral representation
$$
\textsc{E}_{\alpha,1}(-\lambda)=\frac{\sin(\pi
\alpha)}{\pi}\int_{0}^{\infty}\frac{{\mbox
e}^{-\lambda^{1/\alpha}
y}y^{\alpha-1}}{y^{2\alpha}+2y^{\alpha}\cos(\pi
\alpha)+1}dy
$$
which allows one to compute the Mittag Leffler function.
Now, define
\begin{equation}
\textsc{E}^{(k+1)}_{\alpha,1+k\alpha}(-\lambda)=\sum_{l=0}^{\infty}\frac{{(-\lambda)}^{l}}{l!}
\frac{{[k+1]}_{l}} {\Gamma(1+k\alpha+l\alpha )}.\label{altM}
\end{equation}
Theorem 7.1 in James~\cite{JamesLamp} shows that
\begin{equation}
\textsc{E}^{(k+1)}_{\alpha,1+k\alpha}(-\lambda)=\mathbb{E}[{\mbox
e}^{-\lambda^{1/\alpha}X_{\alpha,k\alpha}}]. \label{two}
\end{equation}

In this section set
$L_{\alpha,\lambda}=1/\textsc{E}_{\alpha,1}(-\lambda)$.

\begin{prop}\label{thmLam3} Let, for $0<\alpha<1,$
$(\gamma_{\alpha,\lambda}:\lambda>0)$ denote the family of
densities each defined as
$$
\gamma_{\alpha,\lambda}(t)=L_{\alpha,\lambda}{\mbox
e}^{-\lambda/t^{\alpha}}f_{\alpha}(t).$$ Then, the EPPF's of the
PK$(\rho_{\alpha},\gamma_{\alpha,\lambda})$ family satisfy, for each
fixed $(\alpha,\lambda),$
\begin{enumerate}
\item[(i)]$V_{n,1}=\dfrac{1}{\Gamma(n)}L_{\alpha,\lambda}
\mathbb{E}[\textsc{E}_{\alpha,1}(-{(\B_{1,n-1})}^{\alpha}\lambda)]$.
\item[(ii)] For $k=1,2,\ldots,n$,
$$
V_{n,k}=\frac{\alpha^{k-1}\Gamma(k)}{\Gamma(n)}L_{\alpha,\lambda}
\mathbb{E}[\textsc{E}^{(k+1)}_{\alpha,1+k\alpha}
(-{(\B_{k\alpha,n-k\alpha})}^{\alpha}\lambda)].
$$
\item[(iii)]Equivalently,
$$
V_{n,k}=\frac{\alpha^{k-1}\Gamma(k)}{\Gamma(n-k)}L_{\alpha,\lambda}
\int_{0}^{1}\left[\int_{\mathcal{S}_{k}}\prod_{l=1}^{k}\textsc{E}_{\alpha,1}
(-p^{\alpha}_{l}b^{\alpha}\lambda)dp_{l}\right]b^{k-1}{(1-b)}^{n-k-1}db
$$
where
$\mathcal{S}_{k}=\{(p_{1},\ldots,p_{k}):0<\sum_{i=1}^{k}p_{i}\leq
1\}$.
\item[(iv)] $V_{n,k}$ is described by Proposition~\ref{thmLam2} with
$$
\zeta^{(n,k)}_{\alpha,j}=\frac{{[\sin(\pi \alpha
)]}^{k-j}}{\pi}\int_{0}^{\infty }\mathbb{E}[
{\mbox e}^{-{\lambda^{1/\alpha}\B_{1,n-1}}/{r^{1/\alpha}}}] \frac{\left( r\cos \left( \pi \alpha \right) +1\right) ^{j}r^{\left( k-j\right) -1}%
}{\left[ r^{2}+2r\cos \left( \alpha \pi \right) +1\right] ^{k}}dr.
$$
\end{enumerate}
\end{prop}

\begin{proof}
First note that the relevant $g$ function is
$$
 g_{\lambda}(x)={\mbox e}^{-\lambda^{1/\alpha}x}.
$$
Hence statement~(i) follows from statement~(i) in
Proposition~\ref{thmLam1} and the representation of the Mittag
Leffler function. Statement~(ii) follows from~(\ref{two}) and
Proposition~\ref{thmLam1}. Statement~(iii) follows from
statement~(iii) in Proposition~\ref{thmLam1} and statement~(iv)
follows from Proposition~\ref{thmLam2}. Statement~(iii) is also a
special variation of statement~(v) in Theorem 7.1
in~\cite{JamesLamp}.
\end{proof}

\subsection{Beta-Gamma Classes}\label{sec:BetaGamma} The previous section
identifies a large class of EPPF's that can be computed by using the
random variables $X_{\alpha}$ and $X_{\alpha,k\alpha}.$ Compared
with common random variables that one encounters in standard
probability textbooks, these random variables are rather exotic. In
this section we show, also using results that appear
in~\cite{JamesLamp}, that we may define a large, and clearly
important, class where the EPPF's are expressible in terms of
expectations depending only on beta and gamma random variables.

Recall the identity in~(\ref{key2}), which again may be found in
Pitman~\cite[section 4.2]{Pit06} and Perman, Pitman and
Yor~\cite{PPY92}. Statement~(iii) in Proposition 3.2 in
James~\cite{JamesLamp} uses the result to establish the identity
\begin{equation}
G^{1/\alpha}_{\frac{\theta+\alpha}{\alpha}}\overset{d}=\frac{G_{\theta+\alpha}}{S_{\alpha,\theta+\alpha}}\overset{d}=\frac{G_{\theta+1}}{S_{\alpha,\theta}}
\label{key}
\end{equation}
for $\theta>-\alpha.$ James~\cite[section 3.0.1]{JamesLamp} also
explains how this amounts to a rephrasing of an otherwise equivalent
result in Lemma 6 in Bertoin and Yor~\cite{BerYor}. Now let
$$
1/\Sigma_{\alpha,\theta}=\mathbb{E}[g(G_{\theta/\alpha+1}\B^{\alpha}_{\theta+\alpha,1-\alpha})].
$$
Then $(\ref{key})$ suggests the following class.

\begin{prop}\label{thmLam4} Suppose that for each $0<\alpha<1,$ and
$\theta>-\alpha$,
$$
\gamma_\alpha(t)=
\Sigma_{\alpha,\theta}\mathbb{E}\left[g\left(\frac{G^{\alpha}_{\theta+\alpha}}{t^{\alpha}}\right)\right]
f_{\alpha,\theta}(t).
$$
That is,
$$
h(t)=\frac{\Gamma(\theta+1)}{\Gamma(\theta/\alpha+1)}\Sigma_{\alpha,\theta}t^{-\theta}
\mathbb{E}\left[g\left(\frac{G^{\alpha}_{\theta+\alpha}}{t^{\alpha}}\right)\right].
$$
Then, \PK$(\rho_{\alpha},\gamma_{\alpha})$ has the following
properties.
\begin{enumerate}
\item[(i)]For $n=1,2,\ldots,$
$$V_{n,1}=\frac{\Gamma(\theta+1)}{\Gamma(n+\theta)}\Sigma_{\alpha,\theta}\mathbb{E}[g(G_{\theta/\alpha+1}\B^{\alpha}_{\theta+\alpha,n-\alpha})].$$
\item[(ii)] For $k=1,2,\ldots,n$,
$$
V_{n,k}=\frac{\alpha^{k-1}\Gamma(\theta+1)\Gamma(\theta/\alpha+k)}
{\Gamma(\theta/\alpha+1)\Gamma(n+\theta)} \Sigma_{\alpha,\theta}
\mathbb{E}[g(G_{\theta/\alpha+k}\B^{\alpha}_{\theta+\alpha,n-\alpha})].
$$
\end{enumerate}
\end{prop}

\begin{proof}It suffices to examine the quantity
$$
\mathbb{E}[S^{-\theta}_{\alpha,k\alpha}\B^{\theta}_{k\alpha,n-k\alpha}\,
g(yS^{-\alpha}_{\alpha,k\alpha}\B^{\alpha}_{k\alpha,n-k\alpha})]
$$
for each $k,$ $g$ and fixed $y.$ Now noting the form of the density
of $S_{\alpha,\theta}$ and that of a beta random variable, it
follows that the above expectation can be written as
$$
\frac{\Gamma(n)\Gamma(\theta/\alpha+k)} {\Gamma(n+\theta)\Gamma(k)}
\mathbb{E}[g(yS^{-\alpha}_{\alpha,\theta+k\alpha}\B^{\alpha}_{\theta+k\alpha,n-k\alpha})].
$$
Now replacing $y$ with $G^{\alpha}_{\theta+\alpha}$ shows that
$$
V_{n,k}=\frac{\alpha^{k-1}\Gamma(\theta+1)\Gamma(\theta/\alpha+k)}
{\Gamma(\theta/\alpha+1)\Gamma(n+\theta)}\Sigma_{\alpha,\theta}
\mathbb{E}[g(G^{\alpha}_{\theta+\alpha}S^{-\alpha}_{\alpha,\theta+k\alpha}
\B^{\alpha}_{\theta+k\alpha,n-k\alpha})].
$$
But, using the calculus of beta and gamma random variables,
it follows that since $\theta+\alpha\leq \theta+k\alpha,$
$$
\frac{G_{\theta+\alpha}}{S_{\alpha,\theta+k\alpha}} \overset{d}=
\B_{\theta+\alpha,(k-1)\alpha}\frac{G_{\theta+k\alpha}}{S_{\alpha,\theta+k\alpha}}.
$$
Additionally,
$$
\B_{\theta+\alpha,(k-1)\alpha}\B_{\theta+k\alpha,n-k\alpha}\overset{d}=
\B_{\theta+\alpha,n-\alpha}.
$$
The result is concluded by applying the identity
in~(\ref{key}).
\end{proof}

Note that one can also work with the other equality
in~(\ref{key}). We do not discuss that here.

\subsubsection{Example: Hermite Type}Note that
$$
\mathbb{E}[{\mbox
e}^{-G^{\alpha}_{\theta+\alpha}/t^{\alpha}}]=
\frac{1}{\Gamma(\theta+\alpha)}\int_{0}^{\infty}{\mbox
e}^{-{(\frac{x}{t})}^{\alpha}}x^{\theta+\alpha-1}{\mbox
e}^{-x}dx
$$
which looks like some sort of generalized Hermite function. We can
use this fact along with Proposition~\ref{thmLam4} to calculate an
EPPF, based on the mixing density, which we deliberately write in
the less obvious form,
\begin{equation}
\label{herm}
\gamma_{\alpha,\lambda}(t)=\frac{\Sigma_{\alpha,\theta}(\lambda)f_{\alpha,\theta}(t)}{\Gamma(\theta+\alpha)}\int_{0}^{\infty}{\mbox
e}^{-{\lambda(\frac{x}{t})}^{\alpha}}x^{\theta+\alpha-1}{\mbox
e}^{-x}dx
\end{equation}
where
$$
1/\Sigma_{\alpha,\theta}(\lambda)=
\frac{\Gamma(\theta+1)}{\Gamma(\theta+\alpha)\Gamma(1-\alpha)}
\int_{0}^{1}(1+\lambda
u^{\alpha})^{-(\theta/\alpha+1)}u^{\theta+\alpha-1}{(1-u)}^{1-\alpha-1}du
$$
which is the same as
$$
\mathbb{E}[{\mbox e}^{-\lambda
G_{\theta/\alpha+1}\beta^{\alpha}_{\theta+\alpha,1-\alpha}}]
$$
Here again
$$
g(x)={\mbox e}^{-x},
$$
which leads to the next result.

\begin{prop} Let $\gamma_{\alpha,\lambda}$ be specified
by~(\ref{herm}) for all $0<\alpha<1$ and $\theta>-\alpha.$ Then, the
EPPF of \PK$(\rho_{\alpha},\gamma_{\alpha,\lambda})$ is specified by
$$
V_{n,k}=\Phi^{(k)}_{\alpha,\theta}(\lambda)\int_{0}^{1}(1+\lambda
u^{\alpha})^{-(\theta/\alpha+k)}u^{\theta+\alpha-1}{(1-u)}^{n-\alpha-1}du
$$
where
$$
\Phi^{(k)}_{\alpha,\theta}(\lambda)=\frac{\alpha^{k-1}\Gamma(\theta+1)\Gamma(\theta/\alpha+k)
\Sigma_{\alpha,\theta}(\lambda)}{\Gamma(\theta/\alpha+1)\Gamma(\theta+\alpha)\Gamma(n-\alpha)}.
$$
\end{prop}
\subsection{Comment on Composition Classes} Recall that if one evaluates a
stable subordinator at a positive random time, say $T$,
it is equivalent in distribution to
$$
S_{\alpha}T^{1/\alpha}
$$
having Laplace transform
$$
\mathbb{E}[{\mbox e}^{-S_{\alpha}T^{1/\alpha}}]={\mbox
e}^{-\psi(\lambda^{\alpha})}
$$
where ${\mbox e}^{-\psi(\lambda)}$ is the Laplace transform of $T$
evaluated at $\lambda.$ We see that the two classes in the previous
sections are special cases of this where one chooses $T$ in such a
way where we can work more easily with the stable random variable.
It is clear that many models may be derived using this general idea.
That is, based on
$$
h(t)=\mathbb{E}[g(tT^{1/\alpha})].
$$
Furthermore, since $g$ is arbitrary one can create quite a few
variations of this.

\section{EPPF's via ranked functionals of self-similar Markovian excursions}\label{sec:EPPFMar}
We close this paper by looking at a large class of mixing densities
that probably would not immediately come to mind. This is partly
because, as we shall show, the apparently complex form of them. The
mixing distributions are based on the results of the interesting
paper of Pitman and Yor~\cite{PY01}. In fact, in the general case
the densities do not have known closed forms, which means that
$h(t)$ does not have an explicit form and hence brute force
simulation methods or the calculus of $H$- or $G$-functions cannot
be used directly. However, as we shall show, these are members of
the Beta-Gamma class, and we will indeed be able to get explicit
results by exploiting Corollary 7 in Pitman and Yor~\cite{PY01}. For
clarity, we start off with a simpler yet still challenging case.

\subsection{Example: Hyperbolic Tangent and Kolmogorov's Formula}
First note the famous random variable
$$
M^{\textrm{br}}_{1}:=\max_{0\leq u\leq
1}|B^{\textrm{br}}_{u}|
$$
where $B^{\textrm{br}}_{u}$ denotes standard Brownian bridge on
$[0,1].$ One has the Kolmogorov's formula,
$$
\mathbb{P}(M^{\textrm{br}}_{1}\leq
x)=\sum_{l=-\infty}^{\infty}{(-1)}^{l}e^{-2l^{2}{x}^{2}}.$$
Furthermore, one has the remarkable formula
\begin{equation}
\mathbb{P}(|B_{1}|M^{\textrm{br}}_{1}\leq
y)=\tanh(y),\label{tanform}\end{equation} where $B_{1}$ is a
standard Gaussian random variable independent of
$B^{\textrm{br}}.$ See Pitman~\cite[p. 203]{Pit06} for this
description. We will try to be quite transparent at this point. We
want to use Proposition~\ref{thmLam4} to obtain nice EPPF's
generated by a density related to $M^{\textrm{br}}_{1}.$ The key
to exploiting these results is that
$$
{|B_{1}|}^{2}\overset{d}=2G_{1/2}.
$$
Looking at Proposition~\ref{thmLam4}, we want this to apply for all
$\alpha$, we solve the equation
$$
\frac{\theta}{\alpha}+1=1/2
$$
which leads to
$$
\theta=-\alpha/2,
$$
which is less that $0$ but greater than $-\alpha,$ hence in the
acceptable range of Proposition~\ref{thmLam4}. This leads us to
construct the following mixing density for each $\tau>0,$
\begin{equation}
\label{Kolm} \gamma_{\alpha,\alpha/2,\tau}(t)=
f_{\alpha,-\alpha/2}(t)\mathbb{P}(2{(M^{\textrm{br}}_{1})}^{2}\leq
\tau
t^{\alpha}/G^{\alpha}_{\alpha/2})\Sigma_{\alpha,\alpha/2}(\tau)
\end{equation}
where by~Proposition~\ref{thmLam4} and (\ref{tanform}),
$$
1/\Sigma_{\alpha,\alpha/2}(\tau)=\mathbb{P}(|B_{1}|M^{\textrm{br}}_{1}\leq
\sqrt{\tau}\beta^{-\alpha/2}_{\alpha/2,1-\alpha})=\mathbb{E}[\tanh(\sqrt{\tau}\beta^{-\alpha/2}_{\alpha/2,1-\alpha})].
$$
Now define probabilities
\begin{equation}
\label{Mprob}
p_{n,k}(\tau|\alpha)=\mathbb{P}(2G_{k-1/2}{(M^{\textrm{br}}_{1})}^{2}\leq
\beta^{-\alpha}_{\alpha/2,n-\alpha}\tau).
\end{equation}
This leads to the following result.

\begin{prop}\label{proMar1} Let for $0<\alpha<1,$ and each $\tau>0$,
\PK$(\rho_{\alpha}, \gamma_{\alpha,\alpha/2,\tau})$ denote the
Poisson Kingman distribution with mixing density specified
in~(\ref{Kolm}). In addition, let $p_{n,k}(\tau|\alpha)$ denote
the probabilities defined in~(\ref{Mprob}). Then,
\PK$(\rho_{\alpha}, \gamma_{\alpha,\alpha/2,\tau})$ has the
following properties.
\begin{enumerate}
\item[(i)]For each $n$,
$$V_{n,1}=\frac{\Gamma(1-\alpha/2)}{\Gamma(n-\alpha/2)}
\frac{\mathbb{E}[\tanh(\sqrt{\tau}\beta^{-\alpha/2}_{\alpha/2,n-\alpha})]}{\mathbb{E}[\tanh(\sqrt{\tau}\beta^{-\alpha/2}_{\alpha/2,1-\alpha})]}.
$$
\item[(ii)] For $k=1,2,\ldots,n$,
$$
V_{n,k}=\frac{\alpha^{k-1}\Gamma(1-\alpha/2)\Gamma(k-1/2)}{\Gamma(1/2)\Gamma(n-\alpha/2)}
\frac{p_{n,k}(\tau|\alpha)}{\mathbb{E}[\tanh(\sqrt{\tau}\beta^{-\alpha/2}_{\alpha/2,1-\alpha})]}.
$$
\item[(iii)]For $\tau>0,$ applying~(\ref{forward}) yields the following
recursion,
$$
p_{n+1,k+1}(\tau|\alpha)=\frac{(n-\alpha/2)}{(k-1/2)\alpha}[p_{n,k}(\tau)-p_{n+1,k}(\tau)]+p_{n+1,k}(\tau).
$$
Hence, these probabilities are completely determined by
$$\mathbb{E}[\tanh(\sqrt{\tau}\beta^{-\alpha/2}_{\alpha/2,n-\alpha})]=p_{n,1}(\tau).$$
\end{enumerate}
\end{prop}
\begin{proof}For clarity, it follows from the form of the density
in~(\ref{Kolm}) that Proposition~\ref{thmLam4} applies with
$$
g(x)=\mathbb{P}(2{(M^{\textrm{br}}_{1})}^{2}\leq \tau/x)
$$
and $\theta=-\alpha/2.$ Furthermore,
$$
\mathbb{E}[g(G_{k-1/2}\B^{\alpha}_{\alpha/2,n-\alpha})]=p_{n,k}(\tau|\alpha)
$$
which specializes in the case of $k=1$ to
$$
\mathbb{P}(|B_{1}|M^{\textrm{br}}_{1}\leq
\sqrt{\tau}\beta^{-\alpha/2}_{\alpha/2,n-\alpha})=\mathbb{E}[\tanh(\sqrt{\tau}\beta^{-\alpha/2}_{\alpha/2,n-\alpha})].
$$
due to~(\ref{tanform}).
\end{proof}

\subsection{General result}Next we first sketch out some
of the notation and definitions given in Pitman and Yor~\cite{PY01}.
We ask the interested to consult that work for more details.  Note
instead of $\gamma$ and $\alpha$, used in that work we use $\kappa$
and $\delta.$ The first change is obviously to avoid confusion with
the mixing density. The second as in the previous result is that
$\delta$ will be associated with a stable subordinator, and other
relevant quantities, not directly related to our use of $\alpha.$ We
now paraphrase Pitman and Yor~\cite[p. 366-367]{PY01} which
describes the general random variables we are interested in.

In this section $B:=(B_{t}, t>0)$ denotes a real or vector valued
$\beta$-self similar strong Markov process, with starting state
$0$ which is a recurrent point for $B$. Let $(e_t, 0 \le t \le
V_e)$ denote a generic excursion path, where $V_e$ is the {\em
lifetime} or {\em length} of $e$. Let $F$ be a non-negative
measurable functional of excursions $e$, and let $\kappa >0$. Call
$F$ a {\em $\kappa$-homogeneous functional of excursions of $B$}
if
\begin{equation} \label{Fscale} F( e_t, 0 \le t \le V_e ) =
V_e^\kappa F ( V_e ^{- \beta } e_{u V_e } , 0 \le u \le 1)
. \end{equation}

In particular, we (Pitman and Yor) have in mind the following
functionals $F$: length, maximum height, maximum absolute height,
area, maximum local time. Now we restate their Theorem 3.

\begin{thm}[Pitman and Yor~\cite{PY01}]\label{thmMar1}
Let $F$ be a $\kappa$-homogeneous functional of excursions
of $B$, let $\fst:= F(\Bex)$ for a standard excursion
$\Bex$, and suppose that $\mathbb{E}[
(\fst)^{\delta/\kappa}]$ $< \infty$, for $0<\delta<1.$ Then
the strictly positive values of $F(e)$ as $e$ ranges over
the countable collection of excursions of $\BB$ can be
arranged as a sequence
$$\fbro \ge \fbrt \ge \cdots >  0.$$
Let $\gamal$ be a random variable, independent of $\BB$,
with the gamma$(\delta,1)$ density \begin{equation}
\label{gla}
\mathbb{P}( \gamal  \in dt)/dt = \Gamma(\delta)^{-1} t^{\delta -1 }
e^{-t }\qquad \mbox{for }t > 0.
\end{equation}
Fix $\la > 0$. Then, the joint distribution of the sequence $(\fbrj,
j = 1,2, \ldots)$ is uniquely determined by the equality in
distribution
\begin{equation} \label{mainxx} \left( \mu ( \ggla^{\kappa} \fbrj ) , j
= 1,2, \ldots\right) \overset{d}= (  \tjs , j = 1,2, \ldots
) \mbox{ where } \tjs : = \sum_{i = 1}^j \epi / \epo
\end{equation} for independent standard exponential
variables $\epo, \epone, \eptwo, \ldots$, and $\muu$ is the
function determined as follows by $\delta, \lambda, \kappa$
and the distribution of $\fst$:
\begin{equation} \label{defmu} \muu ( x) := \int_0^ \infty {\delta
\lambda^{-\delta} \over \Gamma(1 - \delta)} t^{-\delta - 1}
e^{-\lambda t } \,\mathbb{P}(\fst > x t^{- \kappa }) \, dt
.
\end{equation}
\end{thm}

As noted by Pitman and Yor~\cite{PY01}, the above result gives a
characterization of the random variables $F^{\textrm{br}}_{j}$ but
it does not provide an explicit expression for their distribution,
as for instance in the special case of $M^{\textrm{br}}_{1}.$
Nonetheless, letting $\mathbb{P}_{\delta}$ denote the law of these
$F^{\textrm{br}}_{j}$, depending on $\delta$ as described above, we
can construct a rich class of EPPF's from these random variables.
Specifically, Theorem 3 and Corollary 7 in Pitman and
Yor~\cite{PY01} along with Proposition~\ref{thmLam4} suggests the
following construction of mixing distributions.

First we solve the equation
$$
\frac{\theta}{\alpha}+1=\delta
$$
which leads to
$$
-\alpha<\theta=(\delta-1)\alpha<0.
$$
Now construct mixing densities as follows. For each
$\alpha,j$ and $\delta$,
\begin{equation}
\label{mixlast}
\gamma_{\alpha,j,\delta}(t)=\left(\Sigma_{j}\right)\times
\mathbb{P}_{\delta}\left(\lambda^{-\kappa}F^{\textrm{br}}_{j}\ge
w{(t/G_{\delta\alpha})}^{\alpha
\kappa}\right)f_{\alpha,(\delta-1)\alpha}(t)
\end{equation}
where
$$
1/\Sigma_{j}=\mathbb{E}\left\{{\left[\frac{\muu (
w\beta^{-\alpha\kappa}_{\alpha\delta,1-\alpha})}{1+\muu (
w\beta^{-\alpha\kappa}_{\alpha\delta,1-\alpha})}\right]}^{j}\right\}=
\mathbb{P}_{\delta}\left(\lambda^{-\kappa}G^{\kappa}_{\delta}F^{\textrm{br}}_{j}\ge
w\beta^{-\alpha\kappa}_{\alpha\delta,1-\alpha}\right)
$$
and $\mu$ defined in~(\ref{defmu}) is a quite general function,
some examples are given in Pitman and Yor~\cite{PY01}. Now define
\begin{equation}
\label{lastprob} p_{n,k}(w|\delta,\alpha,j)=
\mathbb{P}_{\delta}\left(\lambda^{-\kappa}G^{\kappa}_{k+\delta-1}F^{\textrm{br}}_{j}\ge
w\beta^{-\alpha\kappa}_{\alpha\delta,n-\alpha}\right).
\end{equation}
The next result follows from the above discussion combined with
Proposition~\ref{thmLam4} and a bit of algebra.

\begin{thm}\label{thmMar2} Let, for $0<\alpha<1,$ $0<\delta<1,$ $j=1,2,\ldots,$
\PK$(\rho_{\alpha},\gamma_{\alpha,j,\delta})$ denote the Poisson
Kingman distribution with mixing density specified
in~(\ref{mixlast}). In addition, let $p_{n,k}(w|\delta,\alpha,j)$
denote the probabilities defined in~(\ref{lastprob}). Then,
\PK$(\rho_{\alpha},\gamma_{\alpha,j,\delta})$ has the following
properties.
\begin{enumerate}
\item[(i)]For each $n$,
$$V_{n,1}=
\frac{\Gamma(1-(1-\delta)\alpha)}{\Gamma(n-(1-\delta)\alpha)}
\frac{\mathbb{E}\left\{{\left[\dfrac{\muu (
w\B^{-\alpha\kappa}_{\alpha\delta,n-\alpha})}{1+\muu (
w\B^{-\alpha\kappa}_{\alpha\delta,n-\alpha})}\right]}^{j}\right\}}{
\mathbb{E}\left\{{\left[\dfrac{\muu (
w\B^{-\alpha\kappa}_{\alpha\delta,1-\alpha})}{1+\muu (
w\B^{-\alpha\kappa}_{\alpha\delta,1-\alpha})}\right]}^{j}\right\}}.
$$
\item[(ii)] For $k=1,2,\ldots,n$,
$$
V_{n,k}=\frac{\alpha^{k-1}\Gamma(1-(1-\delta)\alpha)\Gamma(k-(1-\delta))}{\Gamma(\delta)\Gamma(n-(1-\delta)\alpha)}
\frac{p_{n,k}(w|\delta,\alpha,j)}{\mathbb{E}\left\{{\left[\dfrac{\muu
( w\B^{-\alpha\kappa}_{\alpha\delta,1-\alpha})}{1+\muu (
w\B^{-\alpha\kappa}_{\alpha\delta,1-\alpha})}\right]}^{j}\right\}}.
$$
\item[(iii)]For $\lambda>0,$ applying~(\ref{forward}) yields the following
recursion,
\begin{eqnarray*}
p_{n+1,k+1}(w|\alpha,\delta,j)&=&\frac{(n-(1-\delta)\alpha)}{(k+\delta-1)\alpha}
\left[p_{n,k}(w|\alpha,\delta,j)-p_{n+1,k}(w|\alpha,\delta,j)\right]\\
&&\qquad +p_{n+1,k}(w|\alpha,\delta,j).
\end{eqnarray*}
Hence, these probabilities are completely determined by
$$
\mathbb{E}\left\{{\left[\frac{\muu (
w\B^{-\alpha\kappa}_{\alpha\delta,n-\alpha})}{1+\muu (
w\B^{-\alpha\kappa}_{\alpha\delta,n-\alpha})}\right]}^{j}\right\}
=p_{n,1}(w|\alpha,\delta,j).
$$
\end{enumerate}
\end{thm}

\begin{proof}The result again follows from
Proposition~\ref{thmLam4}. Now from~(\ref{mixlast}), fixing
$j, \alpha,\delta$ and $w,$
$$
g(x)=\mathbb{P}_{\delta}(\lambda^{-\kappa}F^{\textrm{br}}_{j}\ge
w{x}^{-\kappa}).
$$
Furthermore, $\theta=(\delta-1)\alpha.$ Hence from
Proposition~\ref{thmLam4},
$$
\mathbb{E}[g(G_{k+\delta-1}\B^{\alpha}_{\delta\alpha,n-\alpha})]=p_{n,k}(w|\delta,\alpha,j)
$$
and when $k=1,$ one uses Corollary 7 or Theorem 3 in Pitman
and Yor~\cite{PY01} to obtain
$$
\mathbb{E}\left[g(G_{\delta}\B^{\alpha}_{\delta\alpha,n-\alpha})\right]=
\mathbb{E}\left\{{\left[\frac{\muu (
w\B^{-\alpha\kappa}_{\alpha\delta,n-\alpha})}{1+\muu (
w\B^{-\alpha\kappa}_{\alpha\delta,n-\alpha})}\right]}^{j}\right\}.
$$
\end{proof}

\subsubsection{Bessel Bridges}
As an interesting special case, we now follow the example given in
Pitman and Yor~\cite[p. 375]{PY01}. Let
$$
h_{-\delta}(x)=\frac{I_{-\delta}(x)}{I_{\delta}(x)}
$$
denote the ratio of two modified Bessel functions, which follows
as a special case of~(\ref{defmu}). Then setting $\kappa=1/2,$
$\lambda=1/2,$
$F^{\textrm{br}}_{j}\overset{d}=M^{\textrm{br}}_{j}$, where under
$\mathbb{P}_{\delta}$ are the ranked heights of excursion of a
standard Bessel bridge of dimension $2-2\delta.$ Hence, in this
case,
$$
g(x)=\mathbb{P}_{\delta}(\sqrt{2}M^{\textrm{br}}_{j}\ge
w{x}^{-1/2}),
$$
and we will use, from Pitman and Yor~\cite[p. 375]{PY01},
$$
\mathbb{P}_{\delta}(\sqrt{2G_{\delta}}M^{\textrm{br}}_{j}\ge
w)={(1-h_{-\delta}(w))}^{j}.
$$
Thus Theorem~\ref{thmMar2} specializes as follows.

\begin{prop}Consider the setting in Theorem~\ref{thmMar2} with
$\lambda=1/2$ and $\kappa=1/2$, such that for each $j,$
$F^{\textrm{br}}_{j}\overset{d}=M^{\textrm{br}}_{j}$ are the ranked
heights of excursion of a standard Bessel bridge of dimension
$2-2\delta.$ Then,
$$
V_{n,1}=\frac{\Gamma(1-(1-\delta)\alpha)}{\Gamma(n-(1-\delta)\alpha)}
\frac{\mathbb{E}\left\{{\left[1-h_{-\delta}(w\beta^{-\alpha/2}_{\alpha\delta,n-\alpha})\right]}^{j}\right\}}{
\mathbb{E}\left\{{\left[1-h_{-\delta}(w\beta^{-\alpha/2}_{\alpha\delta,1-\alpha})\right]}^{j}\right\}}
$$
and the $V_{n,k}$ are obtained from $V_{n,1}$ via the recursion
formula based on obvious adjustments in Theorem~\ref{thmMar2}. This
result generalizes Proposition~\ref{proMar1}.
\end{prop}

\section{Appendix: Fox $H$- and Meijer $G$-functions}
\begin{defin}[$H$-Function]\label{AppHdefin}
For integers $m,n,p,q$ such that $0\leq m \leq q,0\leq n\leq p$,
for $a_i,b_j\in \mathbb{C}$ with $\mathbb{C}$, the set of complex
numbers, and for $\alpha_i,\beta_j \in \mathbb{R}_+=(0,\infty)$,
$i=1,2,\ldots,p;j=1,2,\ldots,q$, the $H$-function~\cite{Fox61} is
defined via a Mellin-Barnes type integral~(see~\cite{Pincherle88,
Barnes07, Mellin10, Mainardi03}) on the complex plane in the form
\begin{eqnarray}
H_{p,q}^{m,n}(z) &\equiv& H_{p,q}^{m,n} \left[ z \left|
\begin{matrix} (a_i,\alpha_i)_{1}^{p}\\ \\ (b_j,\beta_j)_{1}^{q}
\end{matrix}\right.\right]\equiv H_{p,q}^{m,n}\left[ z \left|
\begin{matrix} (a_1,\alpha_1),\ldots,(a_p,\alpha_p) \\\\
(b_1,\beta_1),\ldots,(b_q,\beta_q)
\end{matrix}\right.\right]\nonumber\\
&=&\frac{1}{2\pi i} \int_{\mathfrak{L}}
\frac{\displaystyle\prod_{j=1}^m \Gamma(b_j+\beta_js)
\prod_{i=1}^n \Gamma(1-a_i-\alpha_is)}
{\displaystyle\prod_{i=n+1}^p \Gamma(a_i+\alpha_is)
\prod_{j=m+1}^q \Gamma(1-b_j-\beta_js)} z^{-s} ds.\label{AppH}
\end{eqnarray}
Here
$$
z^{-s} =\exp\left[-s(\log|z|+i \arg z)\right],\qquad z\neq
0,\qquad i=\sqrt{-1},
$$
where $\log|z|$ represents the natural logarithm of $|z|$ and
$\arg z$ is not necessarily the principal value. An empty product
in~(\ref{AppH}), if it occurs, is taken to one, and the poles
\begin{equation}\label{poleb}
b_{jl} = \frac{-b_j-l}{\beta_j},\quad j=1,\ldots,m;
l=0,1,2,\ldots,
\end{equation}
of the gamma functions $\Gamma(b_j+\beta_js)$ and the poles
\begin{equation}\label{polea}
a_{ik} = \frac{-a_i-k}{\alpha_i},\quad i=1,\ldots,n;
k=0,1,2,\ldots,
\end{equation}
of the gamma functions $\Gamma(a_i+\alpha_is)$ do not coincide:
$$
\alpha_i(b_j+l) \neq \beta_j(a_i-k-1),\quad
i=1,\ldots,n;j=1,\ldots,m;k,l=0,1,2,\ldots.
$$
$\mathfrak{L}$ in~(\ref{AppH}) is the infinite contour which
separates all the poles $b_{jl}$ in~(\ref{poleb}) to the left and
all the poles $a_{ik}$ in~(\ref{polea}) to the right of
$\mathfrak{L}$, and has one of the following forms:
\begin{enumerate}
    \item[(i)] $\mathfrak{L}=\mathfrak{L}_{-\infty}$ is a
    left loop situated in a horizontal strip starting at
    the point $-\infty+i\varphi_1$ and terminating at the
    point $-\infty+i\varphi_2$ with $-\infty <
    \varphi_1<\varphi_2<+\infty$;
    \item[(ii)] $\mathfrak{L}=\mathfrak{L}_{+\infty}$ is a
    right loop situated in a horizontal strip starting at
    the point $+\infty+i\varphi_1$ and terminating at the
    point $+\infty+i\varphi_2$ with $-\infty <
    \varphi_1<\varphi_2<+\infty$;
    \item[(iii)] $\mathfrak{L}=\mathfrak{L}_{i\gamma\infty}$ is a
    contour starting at
    the point $\gamma-i\infty$ and terminating at the
    point $\gamma+i\infty$, where
    $\gamma\in\mathbb{R}=(-\infty,\infty)$.
\end{enumerate}
\end{defin}
Refer to Kilbas and Saigo~\cite{Kilbas04} for more discussion
about $H$-functions, and in particular, Theorem~1.1 for the
situation in which $H_{p,q}^{m,n}(z)$ defined by~(\ref{AppH})
makes sense.
\begin{defin}[Meijer $G$-Function]\label{AppGdefin}
In general the Meijer $G$-function~\cite{Meijer36, Meijer41,
Meijer46} is defined by the following Mellin-Barnes type integral
on the complex plane:
\begin{eqnarray}
G_{p,q}^{m,n}(z) &\equiv& G_{p,q}^{m,n} \left( z\left|
\begin{matrix} (a_i)_{1}^{p} \\\\ (b_j)_{1}^{q}
\end{matrix}\right.\right) \equiv  G_{p,q}^{m,n} \left( z\left|
\begin{matrix} a_1, \dots, a_p \\\\ b_1, \dots, b_q
\end{matrix}\right.\right) \nonumber\\
&=& \displaystyle\frac{1}{2 \pi i} \displaystyle\int_{\mathfrak L}
\frac{\displaystyle\prod_{j=1}^m \Gamma(b_j + s)
\prod_{j=1}^{n}\Gamma(1 - a_j
-s)}{\displaystyle\prod_{j=n+1}^{p}\Gamma(a_j +
s)\prod_{j=m+1}^{q} \Gamma(1 - b_j - s)} z^{-s} ds,\label{AppG}
\end{eqnarray}
where the contour of integration $\mathfrak L$ is set up to lie
between the poles of $\Gamma(a_i+s)$ and the poles of
$\Gamma(b_j+s)$. The $G$-function is defined under the following
hypothesis
\begin{itemize}
\item $0 \leq m \leq q, 0 \leq n \leq p$ and $p \leq q-1$;
\item $z \neq 0$;
\item no couple of $b_j$, $j = 1,2,\dots,m$ differs by an integer or a
zero;
\item the parameters $a_i \in \mathbb{C}$ and $b_j\in \mathbb{C}$ are so that no
pole of $\Gamma (b_j + s), j = 1,2,\dots,m$ coincide with any pole
of $\Gamma (a_i+s), i = 1,2,\dots,n$;
\item $a_i - b_j
\neq 1,2,3,\dots$ for $i = 1,2,\dots,n$ and $j = 1,2,\dots,m$; and
\item if $p=q$, then the definition makes sense only for
$|z|<1$.
\end{itemize}
\end{defin}
Refer to~\cite[Sec. 5]{Erdelyi53a} for a more thorough discussion
of the $G$-function. Both $H$-functions and $G$-functions are very
general functions whose special cases cover most of the
mathematical functions such as the trigonometric functions, Bessel
functions and generalized hypergeometric functions. Nonetheless,
$G$-functions, but not $H$-functions, are implementable in
\textit{Mathematica} as
\verb"MeijerG[{a1,...,an},{a(n+1),...,ap},{b1,...,bm},"\break\verb"{b(m+1),...,bq},z]".
Comparison between definitions~(\ref{AppHdefin})
and~(\ref{AppGdefin}) reveals that any $G$-function is an
$H$-function, but not vice versa; when $\alpha_i=\beta_j=1$ for
$i=1,\ldots,p$ and $j=1,\ldots,q$,
\begin{equation}\label{AppHG}
  H_{p,q}^{m,n} \left[ z \left|
\begin{matrix} (a_i,1)_{1}^{p} \\\\
(b_j,1)_{1}^{q}
\end{matrix}\right.\right] \equiv G_{p,q}^{m,n} \left( z\left|
\begin{matrix} (a_i)_{1}^{p} \\\\ (b_j)_{1}^{q}
\end{matrix} \right. \right).
\end{equation}
This implies that those $H$-functions which reduce to
$G$-functions can be evaluated in \textit{Mathematica}. Here are
some results, available from Kilbas and Saigo~\cite[Section
2]{Kilbas04}, that follow directly from the definition of the
$H$-function in~(\ref{AppH}).
\begin{property}{\rm The $H$-function is
symmetric in the set of pairs\break
$(a_1,\alpha_1),\ldots,(a_n,\alpha_n)$; in
$(a_{n+1},\alpha_{n+1}),\ldots,(a_p,\alpha_p)$; in
$(b_1,\beta_1),\ldots,(b_m,\beta_m)$ and in
$(b_{m+1},\beta_{m+1}),\ldots,(b_q,\beta_q)$. }
\end{property}
\begin{property}{\rm
If one of the pairs $(a_i,\alpha_i) , i = 1,2,\dots,n$, equals one
of the pairs $(b_j,\beta_j) , j = m+1, \dots, q$, and $n\geq 1,q >
m$, then, for instance,
\begin{equation}\label{AppHlower1}
  H_{p,q}^{m,n} \left[ z \left|
\begin{matrix} (a_i,\alpha_i)_{1}^{p} \\\\
(b_j,\beta_j)_{1}^{q-1},(a_1,\alpha_1)
\end{matrix}\right.\right] =
H_{p-1,q-1}^{m,n-1} \left[ z \left|
\begin{matrix} (a_i,\alpha_i)_{2}^{p} \\\\
(b_j,\beta_j)_{1}^{q-1}
\end{matrix}\right.\right]
\end{equation}
or, if one of the pairs $(a_i,\alpha_i) ,i = n+1, \dots, p$,
equals one of the pairs $(b_j,\beta_j), j = 1,2,\dots,m$ and
$m\geq 1,p>n$, then, for instance,
\begin{equation}\label{AppHlower2}
  H_{p,q}^{m,n} \left[ z \left|
\begin{matrix} (a_i,\alpha_i)_{1}^{p-1},(b_1,\beta_1) \\\\
(b_j,\beta_j)_{1}^{q}
\end{matrix}\right.\right] =
H_{p-1,q-1}^{m-1,n} \left[ z \left|
\begin{matrix} (a_i,\alpha_i)_{1}^{p-1} \\\\
(b_j,\beta_j)_{2}^{q}
\end{matrix}\right.\right].
\end{equation}
}\end{property}
\begin{property}{\rm There holds the relation
\begin{equation}\label{AppHanalytic}
H_{p,q}^{m,n} \left[z\left|
\begin{matrix} (a_i,\alpha_i)_{1}^{p}
\\\\ (b_j,\beta_j)_{1}^{q} \end{matrix}\right.\right] =
H_{q,p}^{n,m} \left[ \frac{1}{z}\left| \begin{matrix}
(1-b_i,\beta_i)_{1}^{q}
\\\\ (1-a_j,\alpha_j)_{1}^{p} \end{matrix}\right.
\right].
\end{equation}
}\end{property}
\begin{property}{\rm
For $y>0$, there holds the relation
\begin{equation}\label{AppHpower}
  H_{p,q}^{m,n} \left[ z \left|
\begin{matrix} (a_i,\alpha_i)_{1}^{p} \\\\
(b_j,\beta_j)_{1}^{q}
\end{matrix}\right.\right] =
y H^{m,n}_{p,q} \left[ z^y \left|
\begin{matrix}
(a_i,k\alpha_i)_{1}^{p}\\\\
(b_j,k\beta_j)_{1}^{q}
\end{matrix}\right.\right].
\end{equation}
}
\end{property}
\begin{property}{\rm
For $\sigma\in\mathbb{C}$, there holds the relation
\begin{equation}\label{AppHshift}
  z^\sigma H_{p,q}^{m,n} \left[ z \left|
\begin{matrix} (a_i,\alpha_i)_{1}^{p} \\\\
(b_j,\beta_j)_{1}^{q}
\end{matrix}\right.\right] =
H^{m,n}_{p,q} \left[ z \left|
\begin{matrix}
(a_i+\sigma \alpha_i,\alpha_i)_{1}^{p}\\\\
(b_j+\sigma \beta_j,\beta_j)_{1}^{q}
\end{matrix}\right.\right].
\end{equation}
}
\end{property}
\begin{property}\label{AppHreduction}{\rm
For $a,b\in\mathbb{C}$, there holds the relations
\begin{equation}\label{AppHzero1}
  H_{p,q}^{m,n} \left[ z \left|
\begin{matrix} (a,0),(a_i,\alpha_i)_{2}^{p} \\\\
(b_j,\beta_j)_{1}^{q}
\end{matrix}\right.\right] =
\Gamma(1-a) H^{m,n-1}_{p-1,q} \left[ z \left|
\begin{matrix}
(a_i,\alpha_i)_{2}^{p}\\\\
(b_j,\beta_j)_{1}^{q}
\end{matrix}\right.\right]
\end{equation}
when $\mbox{Re}(1-a) > 0$ and $n \geq 1$;
\begin{equation}\label{AppHzero2}
  H_{p,q}^{m,n} \left[ z \left|
\begin{matrix} (a_i,\alpha_i)_{1}^{p} \\\\
(b,0),(b_j,\beta_j)_{2}^{q}
\end{matrix}\right.\right] =
\Gamma(b) H^{m-1,n}_{p,q-1} \left[ z \left|
\begin{matrix}
(a_i,\alpha_i)_{1}^{p}\\\\
(b_j,\beta_j)_{2}^{q}
\end{matrix}\right.\right]
\end{equation}
when $\mbox{Re}(b) > 0$ and $m \geq 1$;
\begin{equation}\label{AppHzero3}
  H_{p,q}^{m,n} \left[ z \left|
\begin{matrix} (a_i,\alpha_i)_{1}^{p-1},(a,0) \\\\
(b_j,\beta_j)_{1}^{q}
\end{matrix}\right.\right] =
\frac{1}{\Gamma(a)} H^{m,n}_{p-1,q} \left[ z \left|
\begin{matrix}
(a_i,\alpha_i)_{1}^{p-1}\\\\
(b_j,\beta_j)_{1}^{q}
\end{matrix}\right.\right]
\end{equation}
when $\mbox{Re}(a) > 0$ and $p> n$; and
\begin{equation}\label{AppHzero4}
  H_{p,q}^{m,n} \left[ z \left|
\begin{matrix} (a_i,\alpha_i)_{1}^{p} \\\\
(b_j,\beta_j)_{1}^{q-1},(b,0)
\end{matrix}\right.\right] =
\frac{1}{\Gamma(1-b)} H^{m,n}_{p,q-1} \left[ z \left|
\begin{matrix}
(a_i,\alpha_i)_{1}^{p}\\\\
(b_j,\beta_j)_{1}^{q-1}
\end{matrix}\right.\right]
\end{equation}
when $\mbox{Re}(1-b) > 0$ and $q > m$.}
\end{property}
\begin{property}
The following list shows how it is possible to express some
mathematical functions in terms of the $H$-function:
\begin{eqnarray}
e^z &=& H_{0,1}^{1,0} \left[  -z\left|
\begin{matrix} \overline{\hspace*{0.2in}} \\\\
(0,1)\end{matrix}\right.\right], \qquad \forall
z\nonumber\\
\cos z &=& \frac{1}{\sqrt{\pi}} \,H_{0,2}^{1,0} \left[
\frac{z^2}{4}\left|
\begin{matrix} \overline{\hspace*{0.2in}}\\\\
(0,1),(\frac{1}{2},1)\end{matrix}\right.\right] , \qquad \forall z\nonumber\\
\sin z &=& \frac{2}{\sqrt{\pi}} \,H_{0,2}^{1,0} \left[
\frac{z^2}{4}\left|
\begin{matrix} \overline{\hspace*{0.2in}}\\\\
(0,1),(-\frac{1}{2},1)\end{matrix}\right.\right] , \qquad z \geq 0\nonumber\\
\ln (1+z) &=& H_{2,2}^{1,0} \left[ z\left|
\begin{matrix} (1,1),(1,1)\\\\
(1,1), (0,1)\end{matrix}\right.\right] , \qquad |z|<1\nonumber\\
K_\eta (z) &=& \frac{1}{2} \left(\frac{x}{2}\right)^{-a}
H_{0,2}^{1,0} \left[ \frac{z^2}{2}\left|
\begin{matrix} \overline{\hspace*{0.5in}}\\\\
(\frac{a-\eta}{2},1), (\frac{a+\eta}{2},1)
\end{matrix}\right.\right] , \qquad \forall z\label{AppBessel}
\end{eqnarray}
where the last one is the modified Bessel function of the third
kind.
\end{property}
\begin{property}
The hypergeometric function, when $p \leq q$ or $p=q+1$ with
$0<|z|<1$, can always be expressed in terms of the $H$-function:
\begin{equation}\label{AppGhyper}
_{p}F_{q} \left( a_1,\ldots,a_p; b_1,\ldots,b_q;z \right) =
\frac{\prod_{i = 1}^{p} \Gamma(a_i)}{\prod_{j = 1}^{q}
\Gamma(b_j)} \,H_{p,q+1}^{1,p} \left(-z\left|\begin{matrix}
(1-a_i,1)_{1}^{p}
\\\\ (0,1),(1-b_j,1)_{1}^{q} \end{matrix}\right.\right).
\end{equation}
\end{property}
Refer to~\cite{Kilbas04} for a more exhaustive list of
relationships between the $H$-function and some other mathematical
functions. With $\alpha_i=\beta_j=1$, $i=1,\ldots,p;j=1,\ldots,q$
all the above properties for $H$-functions except
Property~\ref{AppHreduction} yield corresponding properties for
$G$-functions due to the relationship~(\ref{AppHG}). Before
presenting two results from Prudnikov, Brychkov and
Marichev~\cite[P.346,368]{Prudnikov90} concerning calculations of
integrals involving elementary functions and one or two
$G$-functions that are used in this article, we define some more
notation as follows.
\begin{eqnarray}
&m,n,p,q,u,v,w,x=0,1,2,\ldots; r,s=1,2,3,\ldots~(r,s\mbox{ are coprime});\nonumber\\
&0\leq m\leq q,\quad 0\leq n \leq p,\quad 0\leq u\leq x,\quad
0\leq v\leq
w;\nonumber\\
&b^{\ast} = u+v-\dfrac{w+x}{2},\quad c^{\ast} =
m+n-\dfrac{p+q}{2};\label{Gnotation2}\\
&\mu = \displaystyle{\sum_{j=1}^q b_j-\sum_{i=1}^p a_i +
\frac{p-q}{2}+1,\quad\rho = \sum_{j=1}^x d_j-\sum_{i=1}^w
c_i + \frac{w-x}{2}+1};\nonumber\\
&\alpha,\sigma,\xi\in \mathbb{C};\quad \sigma,\xi\neq 0.\nonumber
\end{eqnarray}
For any integer $\ell>0$, define
\begin{equation}\label{Delta}
\Delta(\ell,a) = \left(\frac{a}{\ell}, \frac{a+1}{\ell},\ldots,
\frac{a+\ell-1}{\ell}\right),
\end{equation}
and
\begin{eqnarray}
&&\Delta(\ell,(a_i)_{1}^{p}) =
\left(\frac{a_1}{\ell},\frac{a_2}{\ell},\ldots,\frac{a_p}{\ell},
\frac{a_1+1}{\ell},\frac{a_2+1}{\ell},\ldots,\frac{a_p+1}{\ell},\ldots,\right.\nonumber\\
&&\hspace*{2in}\left.
\frac{a_1+\ell-1}{\ell},\frac{a_2+\ell-1}{\ell},\ldots,\frac{a_p+\ell-1}{\ell}
\right).\label{Delta2}
\end{eqnarray}

\begin{thm}[Prudnikov, Brychkov and
Marichev~\cite{Prudnikov90}] \label{AppG2Gthm}Define notation
in~(\ref{Gnotation2}). If one of the 35 conditions stated
in~\cite[P.346]{Prudnikov90} holds, there holds the relation,
\begin{eqnarray}
&&\int_0^\infty z^{\alpha-1}
   G^{u,v}_{w,x}\left(  \sigma z \left| \begin{matrix} (c_i)_{1}^{w} \\\\
(d_i)_{1}^{x}\end{matrix}\right.\right)
   G_{p,q}^{m,n}\left(  \xi z^{s/r}\left| \begin{matrix} (a_i)_{1}^{p} \\\\
(b_i)_{1}^{q}\end{matrix}\right.\right)
dz\nonumber\\
&& \hspace*{0.5in}=\frac{r^\mu s^{\rho+\alpha(x-w)-1}
\sigma^{-\alpha}}{(2\pi)^{b^{\ast}(s-1)+c^{\ast}(r-1)}}
   G_{rp+sx,rq+sw}^{rm+sv,sn+su}\left(
   \frac{\xi^r r^{r(p-q)}}{\sigma^s s^{s(w-x)}} \left|
   \begin{matrix} (A_i)_{1}^{rp+sx} \\\\
(B_j)_{1}^{rq+sw}\end{matrix}\right.\right)\label{AppG2GthmE}
\end{eqnarray}
where
\begin{eqnarray*}
(A_i)_{1}^{rp+sx}&=&(\Delta(r,a_1),\ldots,\Delta(r,a_n),
\Delta(s,1-\alpha-d_1),\ldots,\\
&&\qquad \qquad\Delta(s,1-\alpha-d_x),
\Delta(r,a_{n+1}),\ldots,\Delta(r,a_p))
\end{eqnarray*}
and
\begin{eqnarray*}
(B_j)_{1}^{rq+sw}&=&(\Delta(r,b_1),\ldots,\Delta(r,b_m),
\Delta(s,1-\alpha-c_1),\ldots, \\
&&\qquad \qquad \Delta(s,1-\alpha-c_w),
\Delta(r,b_{m+1}),\ldots,\Delta(r,b_q)).
\end{eqnarray*}
\end{thm}

\begin{cor}[Prudnikov, Brychkov and
Marichev~\cite{Prudnikov90}]\label{AppGcor1} When $\sigma=y^{-1}$,
$w,x,u=1$, $v=d_1=0$, $c_1=\beta$ in
Theorem~\ref{AppG2Gthm},~(\ref{AppG2GthmE}) reduces to
\begin{eqnarray*}
&&\hspace*{-0.2in}\int_0^y z^{\alpha-1} (y-z)^{\beta-1}
G_{p,q}^{m,n}
\left(  \xi z^{s/r}\left| \begin{matrix} (a_i)_{1}^{p} \\\\
(b_i)_{1}^{q}\end{matrix}\right.\right)
dz\nonumber\\
&& =\frac{r^\mu s^{-\beta} \Gamma(\beta)}
{(2\pi)^{c^{\ast}(r-1)}y^{1-\alpha-\beta}}
G^{rm,rn+s}_{rp+s,rq+s}\left( \frac{\xi^r a^{s}}{r^{r(q-p)}}
\left|\begin{matrix} \Delta(s,1-\alpha),\Delta(r,(a_i)_{1}^{p})\\ \\
\Delta(r,(b_j)_{1}^{q}),\Delta(s,1-\alpha-\beta)
\end{matrix}\right.\right).\label{AppG2Gcor1E}
\end{eqnarray*}
\end{cor}

\end{document}